\documentclass{amsart}
\usepackage[utf8]{inputenc}

\usepackage{amsmath,amssymb,amsthm,mathrsfs,bbm,comment,units,enumitem,xcolor,hyperref}

\usepackage[normalem]{ulem}

\newtheorem{theorem}{Theorem}[section]
\newtheorem{corollary}[theorem]{Corollary}

\newtheorem{proposition}[theorem]{Proposition}

\newtheorem{question}{Question}

\theoremstyle{definition}
\newtheorem{definition}[theorem]{Definition}
\newtheorem{remark}[theorem]{Remark}
\newtheorem{example}[theorem]{Example}

\title{Translation Results for Some Star-Selection Games}

\author{Christopher Caruvana}
\address{School of Sciences\\
Indiana University Kokomo\\
2300 S. Washington Street, Kokomo, IN 46902 USA}
\email{chcaru@iu.edu}
\urladdr{https://chcaru.pages.iu.edu/}

\author{Jared Holshouser}
\address{Department of Mathematics\\
Norwich University\\
158 Harmon Drive, Northfield, VT 05663 USA}
\email{JHolshou@norwich.edu}
\urladdr{https://jaredholshouser.github.io/}

\date{\today}

\subjclass[2010]{91A44, 54D20, 54B20}
\keywords{Star-selection principles, Topological selection games, Pixley-Roy hyperspaces, Uniform spaces}

\begin{document}

\maketitle

\begin{abstract}
    We continue to explore the ways in which high-level topological connections arise from connections between fundamental features of the spaces,
    in this case focusing on star-selection principles in Pixley-Roy hyperspaces and uniform spaces.
    First, we find a way to write star-selection principles as ordinary selection principles,
    allowing us to apply our translation theorems to star-selection games.
    For Pixley-Roy hyperspaces, we are able to extend work of M. Sakai and connect the star-Menger/Rothberger games on the hyperspace
    to the \(\omega\)-Menger/Rothberger games on the ground space.
    Along the way, we uncover connections between cardinal invariants.
    For uniform spaces, we show that the star-Menger/Rothberger game played with uniform covers is equivalent to the
    Menger/Rothberger game played with uniform covers, reinforcing an observation of Lj. Ko{\v{c}}inac.
\end{abstract}

\section{Introduction}

In this paper, we show how the game-theoretic translation techniques developed in \cite{CHContinuousFunctions,CHHyperspaces,CHVietoris}
can be applied to star-selection principles, as studied in
\cite{KocinacStarSelection,KocinacUniform,SakaiStarVersionsMenger,AlamChandra}.
Using a limited-information approach provides a more robust toolkit, and can shed light on various
connections between selection games and selections principles.
We establish a certain relationship between the star-Rothberger (resp. star-Menger) game for the Pixley-Roy hyperspace
of \(X\) with the \(\omega\)-Rothberger (resp. \(\omega\)-Menger) game for \(X\) (see Theorem \ref{thm:FirstBound});
we also establish some equivalences between the uniform-Rothberger and uniform-Menger games on a space
to star versions of those games (see Theorem \ref{thm:MainUniformTheorem}).
We note a game-theoretic characterization of the uniform-Rothberger and uniform-Menger selection principles
for paracompact spaces as an application of the well-known Hurewicz/Pawlikowski theorems (see Corollary \ref{cor:Pawlikowski}).
Along the route to Theorem \ref{thm:FirstBound}, we establish Proposition \ref{prop:AnotherPseudocharacterBound},
a cardinality bound between the so-called
\(k\)-pseudocharacter of \(X\) and the star-Lindel{\"{o}}f degree of the Pixley-Roy hyperspace defined
relative to an ideal of compact subsets of \(X\).
This follows from Proposition \ref{prop:dcBoundsPRAX}, which is a generalization of a cardinality
bound due to Sakai \cite{SakaiCardinal}.
We also expand the applicability of the existing game-theoretic translation machinery
with Theorem \ref{thm:RelationalTranslation} and, with Proposition \ref{prop:SequenceGame},
we compare selection principle properties relative to sequences and sets.

\section{Preliminaries}

We use the word \emph{space} to mean \emph{topological space} and,
unless otherwise stated, all spaces \(X\) are assumed to be Hausdorff.
In general, spaces \(X\) are assumed to be infinite and, when relevant, non-compact.

\subsection{Selection Games and Cover Types} \label{Section:SelectionGames}
In this section, we introduce the various entities
that will be considered in this paper and the general framework for selection games.

\begin{definition}
    For a space \(X\), we let \(\mathscr T_X\) denote the collection of all proper, non-empty open subsets of \(X\).
\end{definition}
\begin{definition}
    Generally, for an open cover \(\mathscr U\) of a topological space \(X\), we say that \(\mathscr U\) is \emph{non-trivial} provided that \(X \not\in \mathscr U\).
    We let \(\mathcal O_X\) denote the collection of all non-trivial open covers of \(X\).
\end{definition}
\begin{definition}
    For a space \(X\) and a class \(\mathcal A\) of closed proper subsets of \(X\), a non-trivial open cover \(\mathscr U\) is an \emph{\(\mathcal A\)-cover} if, for every \(A \in \mathcal A\), there exists \(U \in \mathscr U\) so that \(A \subseteq U\).
    We let \(\mathcal O(X, \mathcal A)\) denote the collection of all \(\mathcal A\)-covers of \(X\).
\end{definition}
Typically, the classes \(\mathcal A\) that will be considered are bases for ideals or bornologies (see \cite{BornologyBook});
that is, collections of closed sets \(\mathcal A\) that cover the space \(X\) with the property that, for every \(A , B \in \mathcal A\), \(A \cup B \in \mathcal A\).
We will also often stipulate that \(\{x\} \in \mathcal A\) for every \(x\in X\).
We will refer to such collections as \emph{ideals of closed sets}.

Two such ideals of particular interest are
\begin{itemize}
    \item
    \([X]^{<\omega}\), the set of finite subsets of \(X\), and
    \item
    \(K(X)\), the set of compact subsets of \(X\).
\end{itemize}
\begin{remark}
    Note that,
    \begin{itemize}
        \item
		if \(\mathcal A = [X]^{<\omega}\), then \(\mathcal O(X,\mathcal A)\) is the collection of all \(\omega\)-covers of \(X\), which will be denoted by \(\Omega_X\).
		\item
		if \(\mathcal A = K(X)\), then \(\mathcal O(X, \mathcal A)\) is the collection of all \(k\)-covers of \(X\), which will be denoted by \(\mathcal K_X\).
    \end{itemize}
\end{remark}

We recall the standard selection principles seen in the literature.
For a primer on selection principles and relevant references, the authors recommend \cite{KocinacSelectedResults,ScheepersSelectionPrinciples,ScheepersNoteMat}.
\begin{definition}
    Let \(\mathcal A\) and \(\mathcal B\) be a classes of sets.
    Then the single- and finite-selection principles are defined, respectively, to be the properties
    \[\mathsf S_1(\mathcal A, \mathcal B) \equiv
    \left(\forall A \in \mathcal A^\omega\right)\left(\exists B \in \prod_{n \in \omega} A_n\right)\ \{B_n : n \in \omega\} \in \mathcal B\]
    and
    \[\mathsf S_{\mathrm{fin}}(\mathcal A, \mathcal B) \equiv
    \left(\forall A \in \mathcal A^\omega\right)\left(\exists B \in \prod_{n \in \omega} [A_n]^{<\omega}\right)\ \bigcup\{B_n : n \in \omega\} \in \mathcal B.\]
\end{definition}

For a space \(X\),
\begin{itemize}
    \item
    \(\mathsf S_{\mathrm{fin}}(\mathcal O_X,\mathcal O_X)\) is known as the \emph{Menger} property.
    \item
    \(\mathsf S_{1}(\mathcal O_X,\mathcal O_X)\) is known as the \emph{Rothberger} property.
    \item
    \(\mathsf S_{\mathrm{fin}}(\Omega_X,\Omega_X)\) is known as the \emph{\(\omega\)-Menger} property.
    \item
    \(\mathsf S_{1}(\Omega_X,\Omega_X)\) is known as the \emph{\(\omega\)-Rothberger} property.
    \item
    \(\mathsf S_{\mathrm{fin}}(\mathcal K_X,\mathcal K_X)\) is known as the \emph{\(k\)-Menger} property.
    \item
    \(\mathsf S_{1}(\mathcal K_X,\mathcal K_X)\) is known as the \emph{\(k\)-Rothberger} property.
\end{itemize}

Selection games arise naturally from the study of selection principles.
Selection games originated as topological games; a history of this development is outlined in Telg{\'{a}}rsky's survey \cite{TelgarskySurvey}
along with a robust list of references.
\begin{definition}
	Given a set \(\mathcal A\) and another set \(\mathcal B\), we define the \emph{finite-selection game}
	\(\mathsf{G}_{\mathrm{fin}}(\mathcal A, \mathcal B)\) for \(\mathcal A\) and \(\mathcal B\) as follows:
	\[
		\begin{array}{c|cccc}
			\mathrm{I} & A_0 & A_1 & A_2 & \ldots \\
			\hline
			\mathrm{II} & \mathcal F_0 & \mathcal F_1 & \mathcal F_2 & \ldots
		\end{array}
	\]
	where \(A_n \in \mathcal A\) and \(\mathcal F_n \in [A_n]^{<\omega}\) for all \(n \in \omega\).
	We declare Two the winner if \(\bigcup\{ \mathcal F_n : n \in \omega \} \in \mathcal B\).
	Otherwise, One wins.
\end{definition}

\begin{definition}
	Similarly, we define the \emph{single-selection game} \(\mathsf{G}_1(\mathcal A, \mathcal B)\) as follows:
	\[
		\begin{array}{c|cccc}
			\mathrm{I} & A_0 & A_1 & A_2 & \ldots \\
			\hline
			\mathrm{II} & x_0 & x_1 & x_2 & \ldots
		\end{array}
	\]
	where each \(A_n \in \mathcal A\) and \(x_n \in A_n\).
	We declare Two the winner if \(\{ x_n : n \in \omega \} \in \mathcal B\).
	Otherwise, One wins.
\end{definition}

\begin{definition}
    We define strategies of various strengths below.
    \begin{itemize}
    \item
    A \emph{strategy for player One} in \(\mathsf{G}_1(\mathcal A, \mathcal B)\) is a function \(\sigma:(\bigcup \mathcal A)^{<\omega} \to \mathcal A\).
    A strategy \(\sigma\) for One is called \emph{winning} if whenever \(x_n \in \sigma\langle x_k : k < n \rangle\) for all \(n \in \omega\), \(\{x_n: n\in\omega\} \not\in \mathcal B\).
    If player One has a winning strategy, we write \(\mathrm{I} \uparrow \mathsf{G}_1(\mathcal A, \mathcal B)\).
    \item
    A \emph{strategy for player Two} in \(\mathsf{G}_1(\mathcal A, \mathcal B)\) is a function \(\tau:\mathcal A^{<\omega} \to \bigcup \mathcal A\).
    A strategy \(\tau\) for Two is \emph{winning} if whenever \(A_n \in \mathcal A\) for all \(n \in \omega\), \(\{\tau(A_0,\ldots,A_n) : n \in \omega\} \in \mathcal B\).
    If player Two has a winning strategy, we write \(\mathrm{II} \uparrow \mathsf{G}_1(\mathcal A, \mathcal B)\).
    \item
    A \emph{predetermined strategy} for One is a strategy which only considers the current turn number.
    We call this kind of strategy predetermined because One is not reacting to Two's moves, they are just running through a pre-planned script.
    Formally it is a function \(\sigma: \omega \to \mathcal A\).
    If One has a winning predetermined strategy, we write \(\mathrm{I} \underset{\mathrm{pre}}{\uparrow} \mathsf{G}_1(\mathcal A, \mathcal B)\).
    \item
    A \emph{Markov strategy} for Two is a strategy which only considers the most recent move of player One and the current turn number.
    Formally it is a function \(\tau:\mathcal A \times \omega \to \bigcup \mathcal A\).
    If Two has a winning Markov strategy, we write \(\mathrm{II} \underset{\mathrm{mark}}{\uparrow} \mathsf{G}_1(\mathcal A, \mathcal B)\).
    \item
    If there is a single element \(A_0 \in \mathcal A\) so that the constant function with value \(A_0\) is a winning strategy for One, we say that One has a \emph{constant winning strategy}, denoted by \(\mathrm{I} \underset{\mathrm{cnst}}{\uparrow} \mathsf{G}_1(\mathcal A, \mathcal B)\).
    \end{itemize}
    These definitions can be extended to \(\mathsf{G}_{\mathrm{fin}}(\mathcal A, \mathcal B)\) in the obvious way.
\end{definition}

\begin{remark} \label{remark:LindelofAndSelection}
    The following are mentioned in \cite[Prop. 15]{ClontzDualSelection} and \cite[Lem. 2.12]{CHVietoris}
    for \(\square \in \{1,\mathrm{fin}\}\).
    \begin{itemize}
        \item
        \(\mathrm{I} \underset{\mathrm{pre}}{\not\uparrow} \mathsf{G}_\square(\mathcal{A},\mathcal{B})\)
        is equivalent to \(\mathsf{S}_\square(\mathcal A, \mathcal B)\).
        \item
        \(\mathrm{I} \underset{\mathrm{cnst}}{\not\uparrow} \mathsf{G}_\square(\mathcal{A},\mathcal{B})\)
        is equivalent to the property that, for every \(A \in \mathcal A\), there is \(\{ x_n : n \in \omega \} \subseteq A\)
        so that \(\{ x_n : n \in \omega \} \in \mathcal B\).
    \end{itemize}
    Note that the property \(\mathrm{I} \underset{\mathrm{cnst}}{\not\uparrow} \mathsf{G}_\square(\mathcal{A},\mathcal{B})\)
    is a Lindel{\"{o}}f-like principle and falls in the category of what Scheepers \cite{ScheepersNoteMat}
    refers to as \emph{Bar-Ilan selection principles}.
\end{remark}

\begin{definition}
  We say that two selection games \(\mathcal G\) and \(\mathcal H\) are \emph{equivalent}, denoted \(\mathcal G \equiv \mathcal H\), if the following hold:
  \begin{itemize}
      \item
      \(\mathrm{II} \underset{\mathrm{mark}}{\uparrow} \mathcal G \iff \mathrm{II} \underset{\mathrm{mark}}{\uparrow} \mathcal H\)
      \item
      \(\mathrm{II} \uparrow \mathcal G \iff \mathrm{II} \uparrow \mathcal H\)
      \item
      \(\mathrm{I} \not\uparrow \mathcal G \iff \mathrm{I} \not\uparrow \mathcal H\)
      \item
      \(\mathrm{I} \underset{\mathrm{pre}}{\not\uparrow} \mathcal G \iff \mathrm{I} \underset{\mathrm{pre}}{\not\uparrow} \mathcal H\)
  \end{itemize}
  If, in addition, \(\mathrm{I} \underset{\mathrm{cnst}}{\not\uparrow} \mathcal G \iff \mathrm{I} \underset{\mathrm{cnst}}{\not\uparrow} \mathcal H\), then we write that \(\mathcal G \rightleftarrows \mathcal H\).
\end{definition}
Note that these notions of game equivalence are more robust than the traditional game equivalence, which only depends on full-information
strategy equivalence.

We now recall a partial ordering on games.
\begin{definition}
  Given selection games \(\mathcal G\) and \(\mathcal H\), we say that \(\mathcal G \leq_{\mathrm{II}} \mathcal H\) if the following implications hold:
  \begin{itemize}
      \item
      \(\mathrm{II} \underset{\mathrm{mark}}{\uparrow} \mathcal G \implies \mathrm{II} \underset{\mathrm{mark}}{\uparrow} \mathcal H\)
      \item
      \(\mathrm{II} \uparrow \mathcal G \implies \mathrm{II} \uparrow \mathcal H\)
      \item
      \(\mathrm{I} \not\uparrow \mathcal G \implies \mathrm{I} \not\uparrow \mathcal H\)
      \item
      \(\mathrm{I} \underset{\mathrm{pre}}{\not\uparrow} \mathcal G \implies \mathrm{I} \underset{\mathrm{pre}}{\not\uparrow} \mathcal H\)
  \end{itemize}
  If, in addition, \(\mathrm{I} \underset{\mathrm{cnst}}{\not\uparrow} \mathcal G \implies \mathrm{I} \underset{\mathrm{cnst}}{\not\uparrow} \mathcal H\), then we write that \(\mathcal G \leq^{+}_{\mathrm{II}} \mathcal H\).
\end{definition}

Observe that the partial ordering \(\mathcal G \leq_{\mathrm{II}} \mathcal H\) effectively asserts that, if Two can win
the game \(\mathcal G\) with a particular level of ``strength,'' then Two can win the game \(\mathcal H\)
with the same kind of strength; this is why we chose to use the subscript \(\mathrm{II}\).
The additional criterion constituting \(\mathcal G \leq^{+}_{\mathrm{II}} \mathcal H\) asserts that the
corresponding Lindel{\"{o}}f-like property, or Bar-Ilan selection principle, is also transferred;
the inequality \(\mathcal G \leq_{\mathrm{II}} \mathcal H\) alone cannot guarantee this since the
implications constituting its definition are dependent upon turn number, in general.

Note that, by Remark \ref{remark:LindelofAndSelection}, if \(\mathsf G_\square(\mathcal A, \mathcal C) \leq_{\mathrm{II}}
\mathsf G_\square(\mathcal B, \mathcal D)\) for \(\square \in \{1,\mathrm{fin}\}\), then we have that, in particular,
\[\mathsf S_\square(\mathcal A, \mathcal C) \implies \mathsf S_\square(\mathcal B, \mathcal D).\]
Also, as the reader can readily verify, for any collections \(\mathcal A\) and \(\mathcal B\),
\[\mathsf G_1(\mathcal A, \mathcal B) \leq^+_{\mathrm{II}} \mathsf G_{\mathrm{fin}}(\mathcal A, \mathcal B).\]

We recall the Translation Theorems, as developed through \cite{CHContinuousFunctions,CHHyperspaces,CHVietoris}.
\begin{theorem}[{\cite[Thm. 2.16]{CHVietoris}}] \label{thm:GeneralTranslation}
    Let \(\mathcal A\), \(\mathcal B\), \(\mathcal C\), and \(\mathcal D\) be collections.
    Suppose there are functions
    \begin{itemize}
        \item \(\overleftarrow{T}_{\mathrm{I},n}:\mathcal B \to \mathcal A\) and
        \item \(\overrightarrow{T}_{\mathrm{II},n}:\left[\bigcup \mathcal A \right]^{<\omega} \times \mathcal B \to \left[\bigcup \mathcal B \right]^{<\omega}\)
    \end{itemize}
    for each \(n \in \omega\) so that
    \begin{enumerate}[label=(P\arabic*)]
        \item \label{FiniteTransA} If \(\mathcal F \in \left[\overleftarrow{T}_{\mathrm{I},n}(B)\right]^{<\omega}\), then \(\overrightarrow{T}_{\mathrm{II},n}(\mathcal F,B) \in [B]^{<\omega}\)
        \item \label{FiniteTransB} If \(\mathcal F_n \in \left[\overleftarrow{T}_{\mathrm{I},n}(B_n)\right]^{<\omega}\) for each \(n \in \omega\)
        and \(\bigcup_{n \in \omega} \mathcal F_n \in \mathcal C\),
        then \[\bigcup_{n \in \omega} \overrightarrow{T}_{\mathrm{II},n}(\mathcal F_n,B_n) \in \mathcal D.\]
    \end{enumerate}
    Then \(\mathsf{G}_{\mathrm{fin}}(\mathcal A, \mathcal C) \leq_{\mathrm{II}} \mathsf{G}_{\mathrm{fin}}(\mathcal B, \mathcal D)\).
    If, in addition, \(\overleftarrow{T}_{\mathrm{I},1} = \overleftarrow{T}_{\mathrm{I},n}\) for all \(n \in \omega\),
    then \(\mathsf{G}_{\mathrm{fin}}(\mathcal A, \mathcal C) \leq^+_{\mathrm{II}} \mathsf{G}_{\mathrm{fin}}(\mathcal B, \mathcal D)\).
\end{theorem}

\begin{corollary}[{\cite[Cor. 2.17]{CHVietoris}}] \label{cor:Translation}
    Let \(\mathcal A\), \(\mathcal B\), \(\mathcal C\), and \(\mathcal D\) be collections.
    Suppose there are functions
    \begin{itemize}
        \item \(\overleftarrow{T}_{\mathrm{I},n} :\mathcal B \to \mathcal A\) and
        \item \(\overrightarrow{T}_{\mathrm{II},n} : \left(\bigcup \mathcal A \right) \times \mathcal B \to \bigcup \mathcal B\)
    \end{itemize}
    for each \(n \in \omega\) so that the following two properties hold.
    \begin{enumerate}[label={(\(\hat{P}\)\arabic*)}]
        \item \label{translationPropI} If \(x \in \overleftarrow{T}_{\mathrm{I},n}(B)\),
        then \(\overrightarrow{T}_{\mathrm{II},n}(x,B) \in B\).
        \item \label{translationPropII} If \(\mathcal F_n \in \left[\overleftarrow{T}_{\mathrm{I},n}(B_n)\right]^{<\omega}\)
        and \(\bigcup_{n \in \omega} \mathcal F_n \in \mathcal C\),
        then \[\bigcup_{n \in \omega} \left\{ \overrightarrow{T}_{\mathrm{II},n}(x,B_n) : x \in \mathcal F_n \right\} \in \mathcal D.\]
    \end{enumerate}
    Then, for \(\square \in \{1,\mathrm{fin}\}\),
    \(\mathsf G_\square(\mathcal A,\mathcal C) \leq_{\mathrm{II}} \mathsf G_\square(\mathcal B, \mathcal D)\).
    If, in addition, \(\overleftarrow{T}_{\mathrm{I},1} = \overleftarrow{T}_{\mathrm{I},n}\) for all \(n \in \omega\), then \(\mathsf{G}_{\square}(\mathcal A, \mathcal C) \leq^+_{\mathrm{II}} \mathsf{G}_{\square}(\mathcal B, \mathcal D)\).
\end{corollary}

\begin{remark} \label{rmk:TranslationConvention}
    In the situation where \(\overleftarrow{T}_{\mathrm{I},1} = \overleftarrow{T}_{\mathrm{I},n}\) for all \(n \in \omega\), we may omit the second subscript and simply write \(\overleftarrow{T}_{\mathrm{I}}\).

    Also, since the conditions in the Translation Theorems relating to the functions \(\overrightarrow{T}_{\mathrm{II},n}\)
    only depend on the range of the corresponding \(\overleftarrow{T}_{\mathrm{I},n}\), we will often partially define
    \(\overrightarrow{T}_{\mathrm{II},n}\) on the relevant inputs.
    In such a case, extending \(\overrightarrow{T}_{\mathrm{II},n}\) beyond the given definition to its entire domain
    is immaterial to the application of the theorem, and thus any mention of such an extension will be omitted.
\end{remark}

\begin{remark} \label{rmk:Lindelof}
    Since, for collections \(\mathcal A\) and \(\mathcal B\), and \(\square \in \{1,\mathrm{fin}\}\),
    \begin{align*}
        \mathrm{II} \underset{\mathrm{mark}}{\uparrow} \mathsf G_\square(\mathcal A,\mathcal B)
        &\implies \mathrm{II} \uparrow \mathsf G_\square(\mathcal A,\mathcal B)\\
        &\implies \mathrm{I} \not\uparrow \mathsf G_\square(\mathcal A,\mathcal B)\\
        &\implies \mathrm{I} \underset{\mathrm{pre}}{\not\uparrow} \mathsf G_\square(\mathcal A,\mathcal B)\\
        &\implies \mathrm{I} \underset{\mathrm{cnst}}{\not\uparrow} \mathsf G_\square(\mathcal A,\mathcal B),
    \end{align*}
    when applying the Translation Theorems to show
    \[\mathsf G_\square(\mathcal A, \mathcal C) \leq_{\mathrm{II}} \mathsf G_\square(\mathcal B, \mathcal D)
    \text{ or }
    \mathsf G_\square(\mathcal A, \mathcal C) \leq^+_{\mathrm{II}} \mathsf G_\square(\mathcal B, \mathcal D),\]
    we may assume that \(\mathrm{I} \underset{\mathrm{cnst}}{\not\uparrow} \mathsf G_\square(\mathcal A, \mathcal C)\);
    that is, we may assume that, as in Remark \ref{remark:LindelofAndSelection}, for every \(A \in \mathcal A\),
    there is some \(\{x_n : n \in \omega\} \subseteq A\) with \(\{x_n :n \in \omega \} \in \mathcal B\).

    In a similar way, we may also assume that \(\mathrm{I} \underset{\mathrm{pre}}{\not\uparrow} \mathsf G_\square(\mathcal A,\mathcal B)\)
    when applying the Translation Theorems to show
    \(\mathsf G_\square(\mathcal A, \mathcal C) \leq_{\mathrm{II}} \mathsf G_\square(\mathcal B, \mathcal D),\)
    which is to say that we may assume \(\mathsf S_\square(\mathcal A,\mathcal C)\) in such a setting.
\end{remark}

As we often make use of the Axiom of Choice, we offer a generalization of the Translation Theorem \ref{cor:Translation}
to make certain applications more readable.
Before we present the generalization, we isolate a technique that will be used.
\begin{remark} \label{rmk:Choice}
    Suppose \(R \subseteq X \times Y\) is so that \(R[x] = \{ y \in Y : (x,y) \in R \} \neq \emptyset\)
    for each \(x \in X\).
    We can define an equivalence relation \(\simeq\) on \(X\) by saying that \(x_1 \simeq x_2\) if
    \(R[x_1] = R[x_2]\).
    We can then define, for \(\mathbf x \in X/{\simeq}\), \(\hat R[\mathbf x] = R[x]\) where \(x \in \mathbf x\).
    Now we apply the Axiom of Choice to select, for each \(\mathbf x \in X/{\simeq}\),
    some \(\hat \gamma (\mathbf x) \in R[\mathbf x]\).
    This \(\hat\gamma\) extends to a function \(\gamma : X \to Y\) where \(\gamma(x_1) = \gamma(x_2)\)
    whenever \(x_1 \simeq x_2\).
\end{remark}

\begin{theorem} \label{thm:RelationalTranslation}
    Suppose \(\mathcal A, \mathcal B, \mathcal C\) and \(\mathcal D\) are collections.
    Suppose there are functions
    \(\overleftarrow{R}_{\mathrm{I},n}:\mathcal B \to \mathcal P(\mathcal A)\) and
    \(\overrightarrow{R}_{\mathrm{II},n}:\bigcup \mathcal A \times \mathcal A \times \mathcal B \to \bigcup \mathcal B\)
    for every \(n \in \omega\) so that
    \begin{enumerate}[label=(R\arabic*)]
        \item \label{RelationTranslationA}
        for all \(B \in \mathcal B\), \(\overleftarrow{R}_{\mathrm{I},n}(B) \neq \emptyset\),
        \item \label{RelationTranslationB}
        whenever \(x \in A \in \overleftarrow{R}_{\mathrm{I},n}(B)\), then \(\overrightarrow{R}_{\mathrm{II},n}(x,A,B) \in B\), and
        \item \label{RelationTranslationC}
        whenever \(F_n \subseteq A_n \in \overleftarrow{R}_{\mathrm{I},n}(B_n)\) are finite so that \(\bigcup_{n\in\omega} F_n \in \mathcal C\), then
        \[
        \left\{\overrightarrow{R}_{\mathrm{II},n}(x,A_n,B_n) : n \in \omega \text{ and } x \in F_n \right\} \in \mathcal D.
        \]
    \end{enumerate}
    Then \(\mathsf G_\square(\mathcal A, \mathcal C) \leq_{\mathrm{II}} \mathsf G_\square(\mathcal B, \mathcal D)\), where \(\square\in\{1,\mathrm{fin}\}\).
    If, in addition, \(\overleftarrow{R}_{\mathrm{I},1} = \overleftarrow{R}_{\mathrm{I},n}\) for all \(n \in \omega\),
    then \(\mathsf G_\square(\mathcal A, \mathcal C) \leq^+_{\mathrm{II}} \mathsf G_\square(\mathcal B, \mathcal D)\).
\end{theorem}
\begin{proof}
    Define \(\overleftarrow{T}_{\mathrm{I},n}:\mathcal B \to \mathcal A\) to be a choice function for \(\overleftarrow{R}_{\mathrm{I},n}\)
    as in Remark \ref{rmk:Choice};
    that is, for each \(B \in \mathcal B\), \(\overleftarrow{T}_{\mathrm{I},n}(B) \in \overleftarrow{R}_{\mathrm{I},n}(B)\) and, for
    \(B_1, B_2 \in \mathcal B\), \(\overleftarrow{R}_{\mathrm{I},n}(B_1) = \overleftarrow{R}_{\mathrm{I},n}(B_2)\) guarantees that \(\overleftarrow{T}_{\mathrm{I},n}(B_1) = \overleftarrow{T}_{\mathrm{I},n}(B_2)\).
    When \(x \in \overleftarrow{T}_{\mathrm{I},n}(B)\), define \(\overrightarrow{T}_{\mathrm{II},n}(x,B)\) by
    \[
    \overrightarrow{T}_{\mathrm{II},n}(x,B) = \overrightarrow{R}_{\mathrm{II},n}\left(x,\overleftarrow{T}_{\mathrm{I},n}(B),B\right).
    \]
    From here, \ref{RelationTranslationB} and \ref{RelationTranslationC} ensure that \(\overleftarrow{T}_{\mathrm{I},n}\) and \(\overrightarrow{T}_{\mathrm{II},n}\) are as required in Corollary \ref{cor:Translation}.

    To address the situation in which \(\overleftarrow{R}_{\mathrm{I},1} = \overleftarrow{R}_{\mathrm{I},n}\) for all \(n \in \omega\),
    simply ensure that \(\overleftarrow{T}_{\mathrm{I},1} = \overleftarrow{T}_{\mathrm{I},n}\) for all \(n \in \omega\) in the
    construction of \(\overleftarrow{T}_{\mathrm{I},n}\).
\end{proof}

\subsection{Star-Selection Games}

The primary goal in this section is to capture star-selection games as selection games as discussed in Section \ref{Section:SelectionGames}.
For a systematic treatment of star-selection principles, the authors recommend \cite{KocinacStarSelection}.
\begin{definition}
    If \(\mathscr U\) is an open cover of \(X\) and \(A \subseteq X\),
    the star of \(A\) relative to \(\mathscr U\) is
    \(\mathrm{St}(A, \mathscr U) = \bigcup \{U \in \mathscr U : U \cap A \neq \emptyset\}\).
    When dealing with singleton sets, we will use
    \(\mathrm{St}(x,\mathscr U)\) in place of \(\mathrm{St}(\{x\},\mathscr U)\).
\end{definition}
\begin{definition}
    For a space \(X\), we define the star-selection principles to be the following:
    \begin{itemize}
    \item
    \(\mathsf S_{\mathrm{fin}}^*(\mathcal O_X, \mathcal B)\) is the property that, for each sequence \(\langle \mathscr U_n : n \in \omega \rangle\) of open covers,
    there are finite subsets \(\mathscr V_n \subseteq \mathscr U_n\) for each \(n \in \omega\) so that
    \(\{\mathrm{St}(\bigcup \mathscr V_n, \mathscr U_n) : n \in \omega\} \in \mathcal B\).
    \item
    \(\mathsf S_1^*(\mathcal O_X, \mathcal B)\) is the property that, for each sequence \(\langle \mathscr U_n : n \in \omega \rangle\) of open covers,
    there are open sets \(U_n \in \mathscr U_n\) for each \(n \in \omega\) so that \(\{\mathrm{St}(U_n, \mathscr U_n) : n \in \omega\} \in \mathcal B\).
    \item
    \(\mathsf{SS}_{\mathcal J}^*(\mathcal O_X, \mathcal B)\) is the property that, for each sequence \(\langle \mathscr U_n : n \in \omega \rangle\) of open covers,
    there are \(J_n \in \mathcal J\) for each \(n \in \omega\) so that \(\{\mathrm{St}(J_n, \mathscr U_n) : n \in \omega\} \in \mathcal B\).
    \end{itemize}
\end{definition}
For any space \(X\),
\begin{itemize}
    \item
    \(\mathsf S^\ast_{\mathrm{fin}}(\mathcal O_X,\mathcal O_X)\) is known as the \emph{star-Menger} property.
    \item
    \(\mathsf S^\ast_{1}(\mathcal O_X,\mathcal O_X)\) is known as the \emph{star-Rothberger} property.
\end{itemize}

As defined, we cannot apply the Translation Theorems to star-selection principles.
We remedy this in the following definition and proposition by finding a way to write
star-selection principles as ordinary selection principles on more complicated sets.

In what follows, when discussing collections of subsets of \(X\), we will identify \(X\) with \([X]^1\), the set of singletons of \(X\),
where the understood bijection is \(x \mapsto \{x\}\).
\begin{definition}
    Let \(X\) be a space.
    \begin{itemize}
        \item If \(\mathscr U\) is an open cover of \(X\) and \(\mathcal J\) is a collection of subsets of \(X\), then \(\mathrm{Cons}(\mathcal J, \mathscr U) = \{ \mathrm{St}(J,\mathscr U) : J \in \mathcal J\}\). Notice that
        \[
        \mathrm{Cons}(\mathscr U, \mathscr U) = \{\mathrm{St}(U, \mathscr U) : U \in \mathscr U\}.
        \]
        \item If \(\mathscr C\) is a collection of covers of \(X\) and \(f:\mathscr C \to \mathcal P(\mathcal P(X))\), then
        \[
        \mathrm{Gal}(f, \mathscr C) = \{\mathrm{Cons}(f(\mathscr U), \mathscr U) : \mathscr U \in \mathscr C\}.
        \]
        If \(f\) is constantly \(\mathcal A\), we write \(\mathcal A\) instead of \(f\).
    \end{itemize}
\end{definition}
Notice that
\[
\mathrm{Gal}(X,\mathcal O_X) = \{\mathrm{Cons}(X,\mathscr U) : \mathscr U \in \mathcal O_X\} = \{\{\mathrm{St}(x,\mathscr U) : x \in X\} : \mathscr U \in \mathcal O_X\},
\]
\[
\mathrm{Gal}(\mathcal A, \mathcal O_X) = \{\mathrm{Cons}(\mathcal A,\mathscr U) : \mathscr U \in \mathcal O_X\} = \{\{\mathrm{St}(A,\mathscr U) : A \in \mathcal A\} : \mathscr U \in \mathcal O_X\},
\]
and
\[
\mathrm{Gal}(\mathrm{id},\mathcal O_X) = \{\mathrm{Cons}(\mathscr U, \mathscr U) : \mathscr U \in \mathcal O_X\} = \{\{\mathrm{St}(U, \mathscr U) : U \in \mathscr U\} : \mathscr U \in \mathcal O_X\}.
\]

We chose the notation \(\mathrm{Cons}\) to stand for \emph{constellation}, a collection of stars;
we chose the notation \(\mathrm{Gal}\) to stand for \emph{galaxy}, a collection of constellations, of sorts.

The reader may readily establish the following.
\begin{proposition} \label{prop:StarSelection}
    The following equivalences hold.
    \begin{itemize}
        \item \(\mathsf S_{\mathrm{fin}}^*(\mathcal O_X, \mathcal B)\) is equivalent to \(\mathsf S_{\mathrm{fin}}(\mathrm{Gal}(\mathrm{id},\mathcal O_X), \mathcal B)\).
        \item \(\mathsf S_1^*(\mathcal O_X, \mathcal B)\) is equivalent to \(\mathsf S_1(\mathrm{Gal}(\mathrm{id}, \mathcal O_X), \mathcal B)\).
        \item \(\mathsf{SS}_{\mathcal J}^*(\mathcal O_X, \mathcal B)\) is equivalent to \(\mathsf S_1(\mathrm{Gal}(\mathcal J,\mathcal O_X), \mathcal B)\).
    \end{itemize}
\end{proposition}

Since we will be interested in selection games, we will let, for notational brevity,
\[\mathsf G_\square^\ast(\mathcal A , \mathcal B) = \mathsf G_\square(\mathrm{Gal}(\mathrm{id},\mathcal A),\mathcal B)\]
and
\[
\mathsf{SG}^\ast_{\mathcal J, \square}(\mathcal A, \mathcal B) = \mathsf G_\square(\mathrm{Gal}(\mathcal J,\mathcal A), \mathcal B),
\]
where \(\square \in \{1,\mathrm{fin}\}\).
When \(\mathcal J = K(X)\), we'll use \(\mathsf{SG}^\ast_{\mathbb K, \square}(\mathcal A, \mathcal B)\) to denote
\(\mathsf{SG}^\ast_{\mathcal J, \square}(\mathcal A, \mathcal B)\).

Observe that the selection principle \(\mathsf{SS}_{\mathcal J}^*(\mathcal A, \mathcal B)\) traditionally appears
in the single-selection context, as noted in Proposition \ref{prop:StarSelection}.
In certain scenarios, like when \(\mathcal J\) forms an ideal of closed subsets of the ground space,
then there is no difference between single- and finite-selections,
as we'll make explicit in Proposition \ref{prop:starSingleVsFinite}.
So the distinction between the selection principles \(\mathsf S_1(\mathrm{Gal}(\mathcal J,\mathcal A), \mathcal B)\)
and \(\mathsf S_{\mathrm{fin}}(\mathrm{Gal}(\mathcal J,\mathcal A), \mathcal B)\) is generally a bit more subtle,
as we illustrate in Examples \ref{ex:starSingleFinite} and \ref{ex:UniformReals}.
This distinction between the two selection principles is why we have chosen to use the notation
\(\mathsf{SG}^\ast_{\mathcal J, \square}(\mathcal A, \mathcal B)\) for the corresponding games.

Recall that a collection \(\mathcal A\) is said to \emph{refine} another collection \(\mathcal B\),
also stated as \(\mathcal A\) \emph{refines} \(\mathcal B\),
if, for every \(A\in \mathcal A\), there exists \(B \in \mathcal B\) so that \(A \subseteq B\).

\begin{proposition} \label{prop:starSingleVsFinite}
    Let \(X\) be a space, \(\mathcal J\) be an ideal of closed sets of \(X\), \(\mathcal Q\) be a collection of open covers of \(X\),
    and \(\mathcal B\) be a collection so that, if \(\mathscr V \in \mathcal B\) and \(\mathscr V\) refines a collection of open sets \(\mathscr U\),
    then \(\mathscr U \in \mathcal B\).
    Then
    \[\mathsf{SG}^\ast_{\mathcal J, 1}(\mathcal Q, \mathcal B) \rightleftarrows
    \mathsf{SG}^\ast_{\mathcal J, \mathrm{fin}}(\mathcal Q, \mathcal B).\]
\end{proposition}
\begin{proof}
    Since \[\mathsf{SG}^\ast_{\mathcal J, 1}(\mathcal Q, \mathcal B) \leq^+_{\mathrm{II}}
    \mathsf{SG}^\ast_{\mathcal J, \mathrm{fin}}(\mathcal Q, \mathcal B)\] is evident,
    we need only show that \[\mathsf{SG}^\ast_{\mathcal J, \mathrm{fin}}(\mathcal Q, \mathcal B) \leq^+_{\mathrm{II}}
    \mathsf{SG}^\ast_{\mathcal J, 1}(\mathcal Q, \mathcal B).\]
    To establish this, we describe how Two can use a winning play in
    \(\mathsf{SG}^\ast_{\mathcal J, \mathrm{fin}}(\mathcal Q, \mathcal B)\) to produce a winning play in
    \(\mathsf{SG}^\ast_{\mathcal J, 1}(\mathcal Q, \mathcal B)\).
    Suppose we have a sequence \(\langle \mathscr U_n : n \in \omega \rangle \in \mathcal Q^\omega\)
    and \(\mathcal F_n \in [\mathcal J]^{<\omega}\) for each \(n \in \omega\) so that
    \[\bigcup_{n\in\omega} \{ \mathrm{St}(J,\mathscr U_n) : J \in \mathcal F_n \} \in \mathcal B.\]
    Since \(\mathcal J\) is assumed to be an ideal of closed sets, \(\bigcup \mathcal F_n \in \mathcal J\)
    for each \(n \in \omega\) and so we see that \(\mathrm{St}\left(\bigcup \mathscr F_n, \mathscr U_n\right)\)
    is a legal sequence of plays by Two in \(\mathsf{SG}^\ast_{\mathcal J, 1}(\mathcal Q, \mathcal B)\).
    By our assumption on \(\mathcal B\) and the fact that
    \[ \bigcup_{J \in \mathcal F_n} \mathrm{St}(J, \mathscr U_n) = \mathrm{St} \left( \bigcup \mathcal F_n , \mathscr U \right)\]
    for each \(n \in \omega\), we also have that
    \[ \left\{ \mathrm{St}\left(\bigcup \mathscr F_n, \mathscr U_n\right) : n \in \omega \right\} \in \mathcal B.\]
    Since the way Two translates the plays between the games is not dependent upon any previous information,
    it is evident that the strategic strength of Two's play is preserved.
\end{proof}
In a similar vein, we obtain the following.
\begin{proposition} \label{prop:SingletonVsFiniteSets}
    For any space \(X\), if \(\mathcal J = [X]^{<\omega}\) and \(\mathcal Q\) is a collection of open covers \(X\),
    \[\mathsf{SG}^\ast_{X, \mathrm{fin}}(\mathcal Q, \mathcal O_X) \rightleftarrows
    \mathsf{SG}^\ast_{\mathcal J, 1}(\mathcal Q, \mathcal O_X).\]
\end{proposition}
\begin{proof}
    The idea of this proof is similar to the proof of Proposition \ref{prop:starSingleVsFinite}.

    Let \(\langle \mathscr U_n : n \in \omega \rangle \in \mathcal Q^\omega\) be arbitrary.

    To see that \[\mathsf{SG}^\ast_{X, \mathrm{fin}}(\mathcal Q, \mathcal O_X)
    \leq_{\mathrm{II}}^+ \mathsf{SG}^\ast_{\mathcal J, 1}(\mathcal Q, \mathcal O_X),\]
    suppose we have \(\mathcal F_n \in [X]^{<\omega}\) so that
    \[\bigcup_{n\in\omega} \{ \mathrm{St}(x, \mathscr U_n) : x \in \mathcal F_n \} \in \mathcal O_X.\]
    Then \(\mathcal F_n \in \mathcal J\) for each \(n \in \omega\) and
    \[\{ \mathrm{St}(\mathcal F_n, \mathscr U_n) : n\in\omega \} \in \mathcal O_X.\]

    To see that \[\mathsf{SG}^\ast_{\mathcal J, 1}(\mathcal Q, \mathcal O_X)
    \leq_{\mathrm{II}}^+ \mathsf{SG}^\ast_{X, \mathrm{fin}}(\mathcal Q, \mathcal O_X),\]
    suppose we have \(\mathcal F_n \in \mathcal J\) so that
    \[\{ \mathrm{St}(\mathcal F_n, \mathscr U_n) : n\in\omega \} \in \mathcal O_X.\]
    Since the target set is \(\mathcal O_X\), it's straightforward to check that
    \[ \bigcup_{n\in\omega} \{ \mathrm{St}(x,\mathscr U_n) : x \in \mathcal F_n \} \in \mathcal O_X\]
    is the corresponding winning play by Two.
\end{proof}
Proposition \ref{prop:SingletonVsFiniteSets} does not generalize to other target sets, as witnessed by the following example.
\begin{example} \label{ex:starSingleFinite}
    For \(X = \mathbb R\) and \(\mathcal J = [\mathbb R]^{<\omega}\),
    \[\mathsf{SG}^\ast_{X, \mathrm{fin}}(\mathcal O_X, \Omega_X) \not\equiv \mathsf{SG}^\ast_{\mathcal J, 1}(\mathcal O_X,\Omega_X).\]
    By \(\sigma\)-compactness of \(\mathbb R\), Two has a winning Markov strategy in
    \(\mathsf{SG}^\ast_{\mathcal J, 1}(\mathcal O_X,\Omega_X)\);
    indeed, given any open cover \(\mathscr U_n\), we can find \(F_n \in [\mathbb R]^{<\omega}\) so that
    \([-n,n] \subseteq \mathrm{St}(F_n, \mathscr U_n)\).

    On the other hand, One has a predetermined winning strategy in \(\mathsf{SG}^\ast_{X, \mathrm{fin}}(\mathcal O_X, \Omega_X)\);
    to see this, let \[\sigma(n) = \left\{ B\left( x , \frac{1}{2^{n+2}} \right) : x \in \mathbb R \right\}.\]
    Note that, for any \(x \in \mathbb R\), \(\mathrm{St}(x, \sigma(n)) \subseteq B\left( x, \frac{1}{2^n} \right)\).
    Therefore, for any sequence \(\langle F_n : n \in \omega \rangle\) of finite sets of \(X\) and any \(n \in \omega\),
    there cannot be any \(x \in F_n\) so that \(\{-2,2\} \subseteq \mathrm{St}(x,\sigma(n))\).
\end{example}

We record the following basic properties of stars, constellations, and galaxies.
\begin{proposition} \label{prop:BasicMonotonicity}
    Let \(X\) be a space, \(\mathcal A, \mathcal B\) be collections of subsets of \(X\), and \(\mathscr U , \mathscr V\)
    be open covers of \(X\).
    \begin{itemize}
        \item
        If \(A, B \subseteq X\) are so that \(A \subseteq B\),
        then \(\mathrm{St}(A, \mathscr U) \subseteq \mathrm{St}(B, \mathscr U)\).
        \item
        If \(\mathscr V\) refines \(\mathscr U\), then, for any \(A \subseteq X\), \(\mathrm{St}(A, \mathscr V) \subseteq \mathrm{St}(A, \mathscr U)\).
        \item
        If \(\mathcal A \subseteq \mathcal B\), then \(\mathrm{Cons}(\mathcal A, \mathscr U) \subseteq \mathrm{Cons}(\mathcal B, \mathscr U)\).
        \item
        If \(\mathcal A\) refines \(\mathcal B\), then \(\mathrm{Cons}(\mathcal A, \mathscr U)\) refines \(\mathrm{Cons}(\mathcal B, \mathscr U)\).
        \item
        If \(\mathscr V\) refines \(\mathscr U\), then \(\mathrm{Cons}(\mathcal A, \mathscr V)\) refines \(\mathrm{Cons}(\mathcal A, \mathscr U)\).
        \item
        If \(\mathcal A\) refines \(\mathcal B\) and \(\mathcal Q\) is a collection of open covers of \(X\),
        then every member of \(\mathrm{Gal}(\mathcal B, \mathcal Q)\) is refined by a member of
        \(\mathrm{Gal}(\mathcal A, \mathcal Q)\).
    \end{itemize}
\end{proposition}

We now record a fact which can be thought of as a monotonicity law (see \cite{ScheepersNoteMat}).
\begin{proposition} \label{prop:Monotonicity}
    Let \(X\) be a set.
    Suppose \(\mathcal A\) and \(\mathcal B\) are families so that, for \(\mathscr F \in \mathcal B\), there exists
    some \(\mathscr G \in \mathcal A\) so that \(\mathscr G\) refines \(\mathscr F\).
    Also suppose that \(\mathcal C\) and \(\mathcal D\) are families so that,
    for every \(\mathscr F \in \mathcal C\) and \(\mathscr G \subseteq \bigcup \mathcal B\), if \(\mathscr F\) refines \(\mathscr G\),
    then \(\mathscr G \in \mathcal D\).
    Then \[\mathsf G_\square(\mathcal A, \mathcal C) \leq_{\mathrm{II}}^+ \mathsf G_\square(\mathcal B, \mathcal D).\]
\end{proposition}
\begin{proof}
    Define a choice function \(\rho : \mathcal B \to \mathcal A\) so that \(\rho(\mathscr F)\) is a refinement of \(\mathscr F\)
    for every \(\mathscr F \in \mathcal B\).
    Now, given \(\mathscr F \in \mathcal B\), define \(\gamma_{\mathscr F} : \rho(\mathscr F) \to \mathscr F\) to be a choice
    function so that \(A \subseteq \gamma_{\mathscr F}(A)\) for every \(A \in \rho(\mathscr F)\).
    We can extend this to a function \[\gamma : \left(\bigcup \mathcal A\right) \times \mathcal B \to \bigcup \mathcal B\]
    so that \(\gamma(A,\mathscr F) = \gamma_{\mathscr F}(A)\) when \(A \in \rho(\mathscr F)\).

    Note that, by the definition, \ref{translationPropI} of Corollary \ref{cor:Translation} is satisfied; that is,
    if \(A \in \rho(\mathscr F)\), then \(\gamma(A,\mathscr F) \in \mathscr F\).

    To finish the proof, note that, if \(\mathbf F_n \in \left[ \rho(\mathscr F_n) \right]^{<\omega}\) for all \(n \in \omega\)
    is so that \(\bigcup_{n\in\omega} \mathbf F_n \in \mathcal C\), then
    \[\bigcup_{n\in\omega} \left\{ \gamma(A, \mathscr F_n) : A \in \mathbf F_n \right\} \in \mathcal D\]
    by our assumptions on \(\mathcal C\) and \(\mathcal D\).

    Therefore, Corollary \ref{cor:Translation} applies to yield that
    \(\mathsf G_\square(\mathcal A, \mathcal C) \leq_{\mathrm{II}}^+ \mathsf G_\square(\mathcal B, \mathcal D).\)
\end{proof}
\begin{corollary} \label{cor:Monotonicity}
    Let \(X\) be a space and \(\mathcal Q_1, \mathcal Q_2\) be collections of open covers of \(X\).
    Additionally suppose that \(\mathcal Q_2\) has the property that, if \(\mathscr U \in \mathcal Q_2\)
    and \(\mathscr V\) is an open cover so that \(\mathscr U\) refines \(\mathscr V\),
    then \(\mathscr V \in \mathcal Q_2\).
    If \(\mathcal A\) and \(\mathcal B\) are collections of closed subsets of \(X\) with \(\mathcal A \subseteq \mathcal B\),
    then
    \[\mathsf{SG}^\ast_{\mathcal A, \square}(\mathcal Q_1, \mathcal Q_2)
    \leq_{\mathrm{II}}^+ \mathsf{SG}^\ast_{\mathcal B, \square}(\mathcal Q_1, \mathcal Q_2).\]
\end{corollary}
\begin{proof}
    By Proposition \ref{prop:BasicMonotonicity}, every element of \(\mathrm{Gal}(\mathcal B, \mathcal Q_1)\)
    is refined by a member of \(\mathrm{Gal}(\mathcal A, \mathcal Q_1)\).
    By our assumption on \(\mathcal Q_2\), Proposition \ref{prop:Monotonicity} applies.
\end{proof}

The following will be used to apply the Translation Theorems under reasonable separation axioms.
\begin{proposition}
    If \(X\) is regular, then every open set is a star of the form \(\mathrm{St}(V,\mathscr U)\).
    If \(X\) is \(T_1\), then every open set is a star of the form \(\mathrm{St}(x,\mathscr U)\).
\end{proposition}
\begin{proof}
    Suppose \(X\) is regular.
    Let \(U \subseteq X\) be open.
    Let \(F \subseteq U\) be closed with non-empty interior.
    Then define \(\mathscr U = \{\mathrm{int}(F), U, X \setminus F\}\).
    Observe that \(\mathrm{St}(V,\mathscr U) = U\).

    Now suppose \(X\) is \(T_1\).
    Let \(U \subseteq X\) be open.
    Let \(x \in X\).
    Then \(X \setminus \{x\}\) is open and \(\mathscr U = \{U,X \setminus \{x\}\}\) is an open cover.
    Now \(\mathrm{St}(x,\mathscr U) = U\).
\end{proof}

\begin{corollary} \label{cor:RegularHasNiceGalaxy}
    If \(X\) is regular, then \(\bigcup \mathrm{Gal}(\mathrm{id},\mathcal O_X) = \mathscr T_X\).
    If \(X\) is \(T_1\), then \(\bigcup \mathrm{Gal}(X,\mathcal O_X) = \mathscr T_X\).
\end{corollary}

The following demonstrates why star-selection principles with respect to \(\omega\)-covers
are not, in general, worth investigating.
\begin{proposition}
    If \(\mathscr U \in \Omega_X\) and \(x \in X\), then \(\mathrm{St}(x,\mathscr U) = X\).
\end{proposition}
\begin{proof}
    Let \(y \in X\). Then since \(\mathscr U \in \Omega_X\), there is a \(V_y \in \mathscr U\) so that \(\{x,y\} \subseteq V_y\). Thus \(V_y \cap \{x\} \neq \emptyset\), and so \(y \in \mathrm{St}(x,\mathscr U)\). So \(X = \mathrm{St}(x,\mathscr U)\).
\end{proof}

\subsection{Sets Versus Sequences in Selection Principles}

We devote this section to a technical discussion of the use of sets versus the use of sequences for selection principles.
In many instances, there is no need to distinguish between the two, but certain results are only true for sequences, and not for sets.
We see this come up with \(\lambda\)-covers, in particular.

\begin{definition}
    Recall that a countably infinite collection of open sets \(\mathscr U\) is a \emph{\(\lambda\)-cover} if,
    for each \(x \in X\), \(\{ U \in \mathscr U : x \in U \}\) is infinite.
    In a similar way, we say that a (potentially uncountable) sequence of open sets \(\langle U_\beta : \beta < \alpha \rangle\)
    is a \emph{\(\lambda\)-cover} if, for each \(x \in X\) and \(\beta < \alpha\), there is a \(\gamma > \beta\) so that \(x \in U_\gamma\).
    We will let \(\Lambda_X\) denote the collection of all \(\lambda\)-covers of \(X\), leaving the set or sequence formulation
    to be evident from context.
\end{definition}

The main distinction here is that a sequence of open sets may have repeats, especially when the open sets have been created by an operation as we will do later in this paper.
In this situation, it is possible for the open sets to form a \(\lambda\)-cover when treated as a sequence, but not when treated as a set.

In light of this issue, we introduce notation for selection principles where the winning condition is in terms of the sequence created
by Player Two as opposed to the set created by them.

\begin{definition}
    Let \(\mathcal A\) be a collection and \(\mathcal B\) be collections. Then the sequence-based selection principles are defined as follows:
    \begin{itemize}
        \item \(\vec{\mathsf S}_1(\mathcal A, \mathcal B)\) holds if and only if, for all sequences \(\langle A_n : n \in \omega \rangle\) of elements of \(\mathcal A\), there are \(x_n \in A_n\) so that \(\langle x_n : n \in \omega \rangle \in \mathcal B\).
        \item \(\vec{\mathsf S}_{\mathrm{fin}}(\mathcal A, \mathcal B)\) holds if and only if, for all sequences \(\langle A_n : n \in \omega \rangle\) of elements of \(\mathcal A\), there are \(F_n \in A_n^{<\omega}\) so that \(F_1 \smallfrown F_2 \smallfrown \cdots \in \mathcal B\).
    \end{itemize}
\end{definition}

We will overload notation when we write these sequence-based selection principles. For instance, \(\vec{\mathsf S}_1(\mathcal O_X, \mathcal O_X)\) will mean that for every sequence of open covers \(\langle \mathscr U_n : n \in \omega \rangle\), there are \(U_n \in \mathscr U_n\) with the property that for each \(x \in X\) there is an \(n \in \omega\) so that \(x \in U_n\).

As with regular selection principles, we can also talk about corresponding selection games and strategies.
We denote the selection games by \(\vec{\mathsf G}_1(\mathcal A, \mathcal B)\) and \(\vec{\mathsf G}_{\mathrm{fin}}(\mathcal A, \mathcal B)\).
In the following proposition, we observe some basic relationships between these variations, establishing why the distinction
is relevant.

\begin{proposition} \label{prop:SequenceGame}
    Let \(X\) be a space and \(\mathcal C \subseteq \mathcal P(\mathscr T_X)\). Then
    \begin{enumerate}[label=(\alph*)]
        \item \label{sequenceGameA}
        \(\vec{\mathsf G}_\square(\mathcal C, \Lambda_X) \leftrightarrows \mathsf G_\square(\mathcal C, \mathcal O_X)\),
        \item \label{sequenceGameB}
        for an ideal \(\mathcal A\) of closed sets of \(X\),
        \(\vec{\mathsf G}_\square(\mathcal C, \mathcal O(X,\mathcal A)) \leftrightarrows \mathsf G_\square(\mathcal C, \mathcal O(X,\mathcal A))\), and
        \item \label{sequenceGameC}
        \(\vec{\mathsf G}_\square(\mathcal C, \Lambda_X) \not\equiv \mathsf G_\square(\mathcal C, \Lambda_X)\).
    \end{enumerate}
\end{proposition}

\begin{proof}
    To establish \ref{sequenceGameA},
    first note that
    \[\vec{\mathsf G}_\square(\mathcal C, \Lambda_X) \leq_{\mathrm{II}}^+ \mathsf G_\square(\mathcal C, \mathcal O_X)\]
    is evident since, assuming Two can play in such a way that they produce a sequence \(\langle U_n : n \in \omega\rangle\)
    which is a \(\lambda\)-cover of \(X\), \(\{ U_n : n \in \omega \}\) is certainly an open cover of \(X\).

    For the reverse inequality,
    choose \(\beta : \omega^2 \to \omega\) to be a bijection so that \(\langle \beta(j,k) : k \in \omega \rangle\)
    is increasing for each \(j \in \omega\).
    Then, assume Two is playing in such a way that they can win \(\mathsf G_\square(\mathcal C, \mathcal O_X)\).
    The main idea here is to apply whatever winning strength Two has on each of the sequences \(\langle \beta(j,k) : k \in \omega \rangle\).
    So consider \(\langle \mathscr U_n : n \in \omega \rangle \in \mathcal C^\omega\) and
    fix some \(n\in \omega\).
    Let \(\langle m , \ell \rangle = \beta^{-1}(n)\) and consider \(\langle \langle m , k \rangle : k < \ell \rangle\).
    Since \(\beta(m,k)\) is increasing as a function of \(k\), we see that \(\beta(m,k) < n\) for all \(k < \ell\).
    Then we can apply whatever level of strength Two is able to use to win on the initial segment consisting of
    \(\langle \mathscr U_{\beta(m,k)} : k < \ell \rangle.\)
    So, for \(m \in \omega\), Two produces \(\langle U_{\beta(m,k)} : k \in \omega \rangle\) so that \(\{ U_{\beta(m,k)}:k\in\omega\} \in \mathcal O_X\).
    Consequently, the sequence \(\langle U_{\beta(m,k)} : m,k \in \omega \rangle\) is a \(\lambda\)-cover of \(X\).

    \ref{sequenceGameB} is evident since, for any sequence of open sets \(\langle U_n : n \in \omega\rangle\),
    \(\langle U_n : n \in \omega\rangle\) has the property
    \[(\forall A \in \mathcal A)(\exists n \in \omega)\ A \subseteq U_n\]
    if and only if \(\{ U_n : n \in \omega \} \in \mathcal O(X,\mathcal A)\).

    For \ref{sequenceGameC}, consider \(\omega\) with the discrete topology.
    Two can win the Rothberger game, but, if One plays the set of singletons, the winning play can never be a \(\lambda\)-cover;
    that is, by \ref{sequenceGameA}, Two can win \(\vec{\mathsf G}_1(\mathcal O_X,\Lambda_X)\), but Two cannot win
    \(\mathsf G_1(\mathcal O_X,\Lambda_X)\).
\end{proof}

We record the following corollary to Proposition \ref{prop:SequenceGame} for use later.
\begin{corollary} \label{cor:starMenger}
    Suppose \(X\) is star-Menger and let \(\langle \mathscr U_n : n \in \omega \rangle\) be a sequence of open covers.
    Then the star-Menger selection principle can be applied to produce a sequence \(\langle \mathscr V_n : n \in \omega\rangle\) where
    \(\mathscr V_n \in [\mathscr U_n]^{<\omega}\) for each \(n \in \omega\) so that, for every \(x \in X\) and \(m \in \omega\),
    there is \(n \geq m\) so that \(x \in \mathrm{St}\left( \bigcup \mathscr V_n , \mathscr U_n \right)\).
\end{corollary}

\section{Star-Selection Games in Pixley-Roy Hyperspaces}

Many connections between selection principles involving \(\omega\)-covers of \(X\) and finite powers of \(X\)
have been established in \cite{GerlitsNagy,JustMillerScheepers,SakaiFunctionSpaces,ScheepersIII}.
By investigating the space of finite sets with the Vietoris topology (see \cite{MichaelHyperspaces}), many analogous results
were established in \cite{CHVietoris}.
Inspired by these connections and results of \cite{SakaiCardinal,SakaiStarVersionsMenger},
we now study Pixley-Roy hyperspaces.

The Pixley-Roy topology was originally defined in \cite{PixleyRoy} on subsets of the real numbers.
We will follow the generalized approach to more general classes of sets found in \cite{DouwenPR}.

\begin{definition}
If \(\mathcal A\) is an ideal of subsets of \(X\), we define the topological space \(\mathrm{PR}_{\mathcal A}(X)\) to be
the set \(\mathcal A\) with basic open sets of the form \[[A,U] = \{B \in \mathcal A : A \subseteq B \subseteq U\}\]
where \(A \in \mathcal A\) and \(U \subseteq X\) is open.
The sets \([A,U]\) form a basis since \(\mathcal A\) is an ideal;
indeed, for \(A, B \in\mathcal A\) and \(U, V \subseteq X\) open,
\[[A \cup B, U \cap V] = [A,U] \cap [B,V].\]
When \(\mathcal A = [X]^{<\omega}\) (resp. \(\mathcal A = K(X)\)), we use \(\mathrm{PR}(X)\)
(resp. \(\mathrm{PR}_{\mathbb K}(X)\)) instead of \(\mathrm{PR}_{\mathcal A}(X)\).
\end{definition}

As noted in \cite{DouwenPR}, for any space \(X\) and any ideal \(\mathcal A\) of subsets of \(X\),
\(\mathrm{PR}_{\mathcal A}(X)\) is \(T_1\) and has a basis consisting of clopen sets.

Now we define the cardinal invariants to be discussed.
For a general introduction to cardinal invariants, see \cite{HodelCardinal}.

\begin{definition}
    For a space \(X\), recall that the Lindel\"{o}f degree of \(X\), \(L(X)\), is the the least cardinal \(\kappa\)
    so that every open cover \(\mathscr U\) has a subcover of size \(\kappa\). Below we define various related cardinals.
    \begin{itemize}
        \item \(hL(X)\), the \emph{hereditary  Lindel\"{o}f degree} of \(X\), is the supremum of the of the Lindel\"{o}f degrees of all subspaces (equivalently, all open subspaces) of \(X\).
        \item \(L_{\mathrm{st}}(X)\), called the \emph{star-Lindel{\"{o}}f degree} of \(X\), denotes the least cardinal \(\kappa\) so that
        every open cover \(\mathscr U\) of \(X\) has a \(\kappa\)-sized subcollection \(\{ U_\alpha : \alpha < \kappa\}\)
        so that \(\{\mathrm{St}(U_\alpha, \mathscr U) : \alpha < \kappa\}\) is an open cover of \(X\).
        \item \(L_\omega(X)\), called the \emph{\(\omega\)-Lindel{\"{o}}f degree} of \(X\), is the least cardinal \(\kappa\) so that, for every \(\omega\)-cover of \(X\), there exists a \(\kappa\)-sized sub-\(\omega\)-cover (a subset which is an \(\omega\)-cover).
        \item \(hL_{\omega}(X)\), called the \emph{hereditary \(\omega\)-Lindel{\"{o}}f degree} of \(X\), is the supremum of the
        \(\omega\)-Lindel\"{o}f degrees of all open subspaces of \(X\).
    \end{itemize}
\end{definition}

Just as we say a space \(X\) is Lindel{\"{o}}f if \(L(X) = \omega\), we say that
\begin{itemize}
    \item
    \(X\) is \emph{hereditarily Lindel{\"{o}}f} if \(hL(X) = \omega\),
    \item
    \(X\) is \emph{star-Lindel{\"{o}}f} if \(L_{\mathrm{st}}(X) = \omega\),
    \item
    \(X\) is \emph{\(\omega\)-Lindel{\"{o}}f} if \(L_\omega(X) = \omega\), and
    \item
    \(X\) is \emph{hereditarily \(\omega\)-Lindel{\"{o}}f} if \(hL_{\omega}(X) = \omega\).
\end{itemize}

\begin{proposition} \label{prop:LindelofChain}
    For any space \(X\) and \(n \in \omega\),
    \[L_{\mathrm{st}}(X) \leq L(X) \leq L(X^n) \leq L\left(\bigcup_{k\in\omega} X^k\right) = L_\omega(X).\]
\end{proposition}
\begin{proof}
    The star-Lindel{\"{o}}f degree is clearly bounded by the Lindel{\"{o}}f degree of \(X\)
    as \(U \subseteq \mathrm{St}(U,\mathscr U)\) for any open cover \(\mathscr U\) of \(X\) and \(U \in \mathscr U\).
    The chain \(L(X) \leq L(X^n) \leq L\left(\bigcup_{k\in\omega} X^k\right)\) holds since \(X\) can be seen as a closed subspace of \(X^n\)
    and \(X^n\) is a closed subspace of \(\bigcup_{k\in\omega} X^k\).
    For the equality \(L\left(\bigcup_{k\in\omega} X^k\right) = L_\omega(X)\), see \cite{GerlitsNagy,JustMillerScheepers}.
\end{proof}

Note however, that these cardinals can be different.
If \(X\) is the Sorgenfrey line, then \(L(X) = \omega < L_\omega(X)\), which holds since
\(X\) is Lindel{\"{o}}f but \(X^2\) is not Lindel{\"{o}}f.
On the other hand, \(\omega_1\) with the order topology is star-Lindel\"{o}f via Fodor's Lemma
but is not Lindel\"{o}f.
So \(L_{\mathrm{st}}(\omega_1) = \omega < L(\omega_1)\).

Inspired by the \(\omega\)-Lindel{\"{o}}f degree, we offer a generalization
to other ideals of closed subsets.
\begin{definition}
    For a space \(X\) and an ideal \(\mathcal A\) of closed subsets of \(X\), we define
    \begin{itemize}
        \item
        \(L_{\mathcal A}(X)\), called the \emph{\(\mathcal A\)-Lindel{\"{o}}f degree} of \(X\),
        to be the least cardinal \(\kappa\) so that, for every \(\mathcal A\)-cover of \(X\),
        there exists a \(\kappa\)-sized sub-\(\mathcal A\)-cover (a subset which is an \(\mathcal A\)-cover).
        \item
        \(hL_{\mathcal A}(X)\), called the \emph{hereditary \(\mathcal A\)-Lindel{\"{o}}f degree} of \(X\),
        to be the supremum of the \(\mathcal A\)-Lindel\"{o}f degrees of all open subspaces of \(X\).
    \end{itemize}
\end{definition}
Note that, when \(\mathcal A = [X]^{<\omega}\), \(L_{\mathcal A}(X) = L_\omega(X)\) and \(hL_{\mathcal A}(X) = hL_\omega(X)\).

\begin{definition}
    Recall that a collection \(\mathscr F\) of subsets of a topological space \(X\) is said to be \emph{discrete} if,
    for every \(x \in X\), there exists an open neighborhood \(U\) of \(X\) so that \(U\) intersects at most one element of \(\mathscr F\).
    Let \(dc(X)\), called the \emph{discrete cellularity} of \(X\), denote the supremum over all cardinals \(\kappa\)
    so that there exists a \(\kappa\)-sized discrete family \(\mathscr U\) of open sets.
\end{definition}

\begin{definition}
    The \emph{pseudocharacter} \(\psi(X)\) of a space \(X\) is the least cardinal \(\kappa\) so that every point is the \(\kappa\)-length
    intersection of open sets.
    Analogously, the \emph{\(k\)-pseudocharacter} \(\psi_k(X)\) of a space \(X\) is the least cardinal \(\kappa\) so that
    every compact subset of \(X\) is the \(\kappa\)-length intersection of open sets.
\end{definition}

The following cardinal inequalities are inspired by various results of \cite{SakaiCardinal,SakaiStarVersionsMenger}.
\begin{proposition} \label{prop:StarLindelofToDiscreteCell}
    If \(X\) has a basis of clopen sets, then \(dc(X) \leq L_{\mathrm{st}}(X)\).
\end{proposition}
\begin{proof}
    Let \(\kappa\) be cardinal so that \(dc(X) > \kappa\).
    We will show that \(L_{\mathrm{st}}(X) > \kappa\) as well, which in turn means that \(L_{\mathrm{st}}(X) \geq dc(X)\).
    By our assumption that \(dc(X) > \kappa\), there is a cardinal \(\lambda > \kappa\) and a collection of non-empty open sets
    \(\{U_\alpha : \alpha < \lambda\}\) which is discrete (and so also pairwise disjoint).
    As we have a basis of clopen sets, we can assume the \(U_\alpha\) are clopen.
    Set \(W = X \setminus \bigcup_{\alpha < \lambda}U_\alpha\).
    To show that \(W\) is open, let \(y \in W\).
    Then there is an open set \(V\) so that \(y \in V\) and \(V\) meets at most one \(U_\alpha\).
    Then \(y \in V \setminus U_\alpha\), and \(V \setminus U_\alpha \subseteq W\), so \(W\) is open.

    Let \(\mathscr U = \{U_\alpha : \alpha < \lambda\} \cup \{W\}\) and note that this forms an open cover of \(X\).
    Suppose \(\mathscr V\) is a \(\kappa\)-sized sub-collection of \(\mathscr U\).
    Since the \(U_\alpha\) are disjoint, \(\mathrm{St}(U_\alpha,\mathscr U) = U_\alpha\) and \(\mathrm{St}(W,\mathscr U) = W\) as well.
    Then
    \[
    \bigcup_{V \in \mathscr V} \mathrm{St}\left( V , \mathscr U \right)
    = \mathrm{St}\left( \bigcup\mathscr V , \mathscr U \right) = \bigcup \mathscr V \neq X,
    \]
    which implies that \(L_{\mathrm{st}}(X) > \kappa\).
\end{proof}
\begin{corollary} \label{cor:PRAXStarLindelof}
    For any space \(X\) and any ideal \(\mathcal A\) of compact subsets of \(X\),
    \[
    dc(\mathrm{PR}_{\mathcal A}(X)) \leq L_{st}(\mathrm{PR}_{\mathcal A}(X)).
    \]
\end{corollary}
\begin{proof}
    Since \(\mathrm{PR}_{\mathcal A}(X)\) has a basis of clopen sets, we can apply Proposition \ref{prop:StarLindelofToDiscreteCell}.
\end{proof}

Propositions \ref{prop:dcBoundsPRAX} and \ref{prop:hereditaryABoundsHereditary}
offer a generalization of \cite[Thm. 3.10]{SakaiCardinal}.
\begin{proposition} \label{prop:dcBoundsPRAX}
    For any space \(X\) and any ideal \(\mathcal A\) of compact subsets of \(X\), \[hL_{\mathcal A}(X) \leq dc(\mathrm{PR}_{\mathcal A}(X)).\]
\end{proposition}
\begin{proof}
    We first show that, for open \(U \subseteq X\), \(\mathrm{PR}_{\mathcal A}(U)\) is a clopen subspace of \(\mathrm{PR}_{\mathcal A}(X)\).
    \(\mathrm{PR}_{\mathcal A}(U)\) is clearly open as the union of basic open sets.
    To see that \(\mathrm{PR}_{\mathcal A}(U)\) is closed, let \(A \not\in \mathrm{PR}_{\mathcal A}(U)\).
    That means that \(A \not\subseteq U\), so there is some \(x \in A \setminus U\).
    Then \([\{x\} , X]\) is a neighborhood of \(A\) that is disjoint from \(\mathrm{PR}_{\mathcal A}(U)\).

    We now show that \(L_{\mathcal A}(X) \leq dc(\mathrm{PR}_{\mathcal A}(X)).\)
    Assume \(\mathscr U = \{U_\alpha : \alpha < \lambda\}\) is an \(\mathcal A\)-cover of \(X\) consisting of non-empty sets.
    For \(\alpha < \lambda\), let
    \[
    V_\alpha = \mathrm{PR}_{\mathcal A}(U_\alpha) \setminus \bigcup_{\beta < \alpha} V_\beta.
    \]
    Clearly, the \(V_\alpha\) are disjoint.
    We claim they are clopen as well.

    Note that \(V_0 = \mathrm{PR}_{\mathcal A}(U_0)\), which is clopen.
    So, given \(\alpha > 0\), suppose that \(V_\beta\) is clopen for all \(\beta < \alpha\).
    It is immediate then that \(V_\alpha\) is closed.
    To see that \(V_\alpha\) is open, let \(A \in V_\alpha\).
    Then \([A, U_\alpha]\) is a neighborhood of \(A\) disjoint from \(\bigcup_{\beta < \alpha} V_\beta\).
    To see this, let \(B \in [A, U_\alpha]\) and \(\beta < \alpha\).
    Note that
    \[\bigcup_{\gamma \leq \beta} V_\gamma = \bigcup_{\gamma \leq \beta} \mathrm{PR}_{\mathcal A}(U_\gamma).\]
    Since \(A \not\in \bigcup_{\gamma < \alpha} V_\gamma\), we see that \(A \not\subseteq U_\beta\).
    Hence, \(B \not\subseteq U_\beta\), so \(B \not\in V_\beta\).

    By a similar argument as what's above, \(\{ V_\alpha : \alpha < \lambda \}\) is a discrete family of clopen sets.
    So, by the discrete cellularity, there are \(\alpha_\xi < \lambda\) for \(\xi < \delta := dc(\mathrm{PR}_{\mathcal A}(X))\) so that
    \(V_{\alpha_\xi} \neq \emptyset\) for all \(\xi < \delta\) and \(V_\beta = \emptyset\) for all \(\beta \notin \{\alpha_\xi : \xi < \delta\}\).
    Now, \(\{ V_{\alpha_\xi} : \xi < \delta\}\) is a cover of \(\mathrm{PR}_{\mathcal A}(X)\).
    To see this, let \(A \in \mathcal A\) and take \(\alpha < \lambda\) to be minimal so that \(A \subseteq U_\alpha\).
    Then \(A \in V_\alpha\) which means \(\alpha \in \{ \alpha_\xi : \xi < \delta\}\).
    Hence, \(\{ U_{\alpha_\xi} : \xi < \delta \}\) is an \(\mathcal A\)-cover of \(X\).

    We finally show that \(hL_{\mathcal A}(X) \leq dc(\mathrm{PR}_{\mathcal A}(X)).\)
    Let \(V \subseteq X\) be open.
    Then \(\mathrm{PR}_{\mathcal A}(V)\) is clopen, and so \(dc(\mathrm{PR}_{\mathcal A}(V)) \leq dc(\mathrm{PR}_{\mathcal A}(X))\).
    Thus \(L_{\mathcal A}(V) \leq dc(\mathrm{PR}_{\mathcal A}(X))\) by the same proof as above.
\end{proof}
\begin{proposition} \label{prop:hereditaryABoundsHereditary}
    For a non-compact space \(X\) and any ideal \(\mathcal A\) of compact subsets of \(X\), \[hL(X) \leq hL_{\mathcal A}(X).\]
\end{proposition}
\begin{proof}
    Suppose \(\mathscr U\) is an open cover of an open set \(V \subseteq X\).
    Let \[\mathscr V = \left\{ \bigcup \mathcal F : \mathcal F \in \left[\mathscr U\right]^{<\omega}\right\}.\]
    Notice that, since \(\mathcal A\) consists of compact sets, \(\mathscr V\) is an \(\mathcal A\)-cover of \(V\).
    That means we can find a sub-\(\mathcal A\)-cover \(\mathscr V'\) so that \(|\mathscr V'| \leq hL_{\mathcal A}(X)\).
    Now for each \(V \in \mathscr V'\), choose \(\mathcal F_V \subseteq \mathscr U\) which is finite so that \(V = \bigcup \mathcal F_V\).
    Let \(\mathscr U' = \bigcup_{V \in \mathscr V'}\mathcal F_V\).
    Then \(\mathscr U'\) is a cover of \(V\) and
    \[
    |\mathscr U'| \leq |\mathscr V'| \cdot \aleph_0 \leq hL_{\mathcal A}(X) \cdot \aleph_0.
    \]
    The rightmost part is equal to \(hL_{\mathcal A}(X)\) provided that it is infinite.
\end{proof}

\begin{proposition} \label{prop:HereditaryLindelofBoundsPseudocharacter}
    If \(X\) is regular, \(\psi_k(X) \leq hL(X)\).
\end{proposition}
\begin{proof}
    Let \(K \subseteq X\) be compact and consider \[\mathscr V = \{ F \subseteq X : K \subseteq \mathrm{int}(F) \text{ and } F \text{ is closed}\}.\]
    Notice that \(\{ X \setminus F : F \in \mathscr V \}\) is a cover of \(X \setminus K\)
    by regularity.
    Using the hereditary Lindel{\"{o}}f degree, there is an \(hL(X)\)-sized collection \(\{X \setminus F_\alpha : \alpha < hL(X)\}\)
    so that \(X \setminus K = \bigcup_{\alpha < hL(X)} (X \setminus F_\alpha)\).
    Let \(U_\alpha = \mathrm{int}(F_\alpha)\) and notice that \(K \subseteq U_\alpha\) for all \(\alpha < hL(X)\).
    Clearly, \(K \subseteq \bigcap_{\alpha < hL(X)} U_\alpha\).
    Now if \(y \not\in K\), then \(y \in X \setminus F_\alpha\) for some \(\alpha\).
    Thus \(y \notin U_\alpha\), so \(y \notin \bigcap_{\alpha < \kappa} U_\alpha\).
    Therefore \(K = \bigcap_{\alpha < hL(X)} U_\alpha\).
\end{proof}
\begin{proposition} \label{prop:AnotherPseudocharacterBound}
    If \(X\) is a non-compact regular space and \(\mathcal A\) is an ideal of compact subsets of \(X\), then
    \[\psi_k(X) \leq L_{st}(\mathrm{PR}_{\mathcal A}(X)).\]
\end{proposition}
\begin{proof}
    Since \(X\) is regular, \(\psi_k(X) \leq hL(X)\)
    by Proposition \ref{prop:HereditaryLindelofBoundsPseudocharacter}.
    Then Propositions \ref{prop:hereditaryABoundsHereditary}, \ref{prop:dcBoundsPRAX} and Corollary \ref{cor:PRAXStarLindelof}
    show that
    \[
    hL(X) \leq hL_{\mathcal A}(X) \leq dc(\mathrm{PR}_{\mathcal A}(X)) \leq L_{st}(\mathrm{PR}_{\mathcal A}(X)).
    \]
    Therefore, \(\psi_k(X) \leq L_{st}(\mathrm{PR}_{\mathcal A}(X))\).
\end{proof}

The following captures \cite[Thm. 4.12]{SakaiStarVersionsMenger} (see also \cite[Thm. 3.4 (2)]{KocinacStarSelection}), but extends
it to the other strategy types considered in this paper.
\begin{theorem} \label{thm:FirstBound}
    Assume \(X\) is regular. Then, for \(\square \in \{1,\mathrm{fin}\}\),
    \[
    \mathsf G_\square^*(\mathcal O_{\mathrm{PR}(X)},\mathcal O_{\mathrm{PR}(X)}) \leq_{\mathrm{II}} \mathsf G_\square(\Omega_X,\Omega_X).
    \]
\end{theorem}

\begin{proof}
    As noted in Remark \ref{rmk:Lindelof}, we may assume that \(\mathrm{PR}(X)\) is star-Menger.
    Then \(\mathrm{PR}(X)\) is also star-Lindel{\"{o}}f which,
    by Proposition \ref{prop:AnotherPseudocharacterBound}, implies that every singleton of \(X\) is a \(G_\delta\).
    So, for \(x \in X\), we define open sets \(G_{x,n}\) so that \(\{x\} = \bigcap \{ G_{x,n} : n \in \omega \}\).
    Without loss of generality, we suppose the sequence \(\langle G_{x,n} : n \in \omega \rangle\) is descending.
    Then, for \(F \in [X]^{<\omega}\) and \(n \in \omega\), we define \(G_{F,n} = \bigcup \{ G_{x,n} : x \in F \}\).
    Note that the sequence \(\langle G_{F,n} : n \in \omega\rangle\) is also descending.

    We define a choice function \(\gamma : \Omega_X \times \mathrm{PR}(X) \to \mathscr T_X\) by choosing
    \(\gamma(\mathscr U, F) \in \mathscr U\) to be so that \(F \subseteq \gamma(\mathscr U , F)\).
    For \(\mathscr U \in \Omega_X\), \(n \in \omega\), and \(F \in \mathrm{PR}(X)\), let
    \[\mathbf W(F,\mathscr U,n) = [F,\gamma(\mathscr U,F) \cap G_{F,n}].\]
    Then, for \(\mathscr U \in \Omega_X\) and \(n\in\omega\), define
    \[\mathscr V_{\mathscr U,n} = \{ \mathbf W(F,\mathscr U,n) : F \in \mathrm{PR}(X) \}.\]

    Then we can define \(\overrightarrow{T}_{\mathrm{I},n} : \Omega_X \to \mathrm{Gal}(\mathrm{id},\mathcal O_{\mathrm{PR}(X)})\) by the rule
    \[\overrightarrow{T}_{\mathrm{I},n}(\mathscr U) = \mathrm{Cons}(\mathscr V_{\mathscr U, n},\mathscr V_{\mathscr U, n}).\]

    Fix \(\mathscr U \in \Omega_X\) and \(n \in \omega\).
    Then, for \(W \in \mathrm{Cons}(\mathscr V_{\mathscr U, n},\mathscr V_{\mathscr U, n})\),
    consider \[\{ F \in \mathrm{PR}(X) : W = \mathrm{St}(\mathbf W(F,\mathscr U,n) , \mathscr V_{\mathscr U,n}) \}.\]
    We can define a choice function \(\chi_{\mathscr U, n} : \mathrm{Cons}(\mathscr V_{\mathscr U, n},\mathscr V_{\mathscr U, n}) \to \mathrm{PR}(X)\)
    to be so that \(\chi_{\mathscr U, n}(W) \in \mathrm{PR}(X)\) satisfies
    \[W = \mathrm{St}(\mathbf W(\chi_{\mathscr U, n}(W) , \mathscr U, n) , \mathscr V_{\mathscr U, n}).\]

    Since \(\mathrm{PR}(X)\) is regular, Corollary \ref{cor:RegularHasNiceGalaxy} asserts that
    \(\bigcup \mathrm{Gal}(\mathrm{id}, \mathcal O_{\mathrm{PR}(X)}) = \mathscr T_{\mathrm{PR(X)}}\).
    So we define \(\overrightarrow{T}_{\mathrm{II},n} : \mathscr T_{\mathrm{PR}(X)} \times \Omega_X \to \mathscr T_X\),
    in light of Remark \ref{rmk:TranslationConvention}, as follows:
    if \(W \in \mathrm{Cons}(\mathscr V_{\mathscr U, n},\mathscr V_{\mathscr U, n})\),
    let \[\overrightarrow{T}_{\mathrm{II},n}(W,\mathscr U) = \gamma(\mathscr U, \chi_{\mathscr U,n}(W)).\]

    Now, using Corollary \ref{cor:starMenger}, assume we have sequences \(\langle \mathscr U_n : n \in \omega \rangle \in \Omega_X^\omega\) and
    \(\langle \mathscr F_n : n \in \omega \rangle\) so that
    \(\mathscr F_n \in \left[\overrightarrow{T}_{\mathrm{I},n}(\mathscr U_n)\right]^{<\omega}\) for each \(n \in \omega\) and, for every \(F \in\mathrm{PR}(X)\) and
    every \(m \in \omega\), there exists \(n \geq m\) and \(W \in \mathscr F_n\) so that \(F \in W\).
    %Then, set \(\mathscr G_n = \{ \chi_{\mathscr U_n,n}(W) : W \in \mathscr F_n \}\).
    To show that \[\bigcup_{n\in\omega}\left\{ \overrightarrow{T}_{\mathrm{II},n}(W,\mathscr U_n) : W \in \mathscr F_n \right\} \in \Omega_X,\]
    let \(F \in [X]^{<\omega}\) be arbitrary.
    Then we can let \(\langle \ell_n : n \in \omega \rangle\) be cofinal in \(\omega\)
    and, for each \(n\in\omega\), choose \(W_n \in \mathscr F_{\ell_n}\) so that
    \(F \in W_{n}\) for all \(n \in \omega\).
    To simplify notation, let \(F_n = \chi_{\mathscr U_{\ell_n},\ell_n}(W_n)\).
    Then, for each \(n\in\omega\), there is \(E_n \in \mathrm{PR}(X)\) so that
    \[\mathbf W(E_n , \mathscr U_{\ell_n}, \ell_n) \cap
    \mathbf W\left( F_n, \mathscr U_{\ell_n} , \ell_n\right)
    \neq \emptyset\]
    and \(F \in \mathbf W(E_n , \mathscr U_{\ell_n}, \ell_n)\).
    Observe that \(E_n \subseteq F\) for each \(n \in \omega\) so,
    as the map \(n \mapsto E_n\), \(\omega \to \wp(F)\), is into a finite set,
    there must be some cofinal \(\langle k_n : n \in \omega\rangle\) in
    \(\omega\) so that \(\langle E_{k_n} : n \in \omega \rangle\) is constant.
    Hence, let \(E \in \mathrm{PR}(X)\) be so that \(E = E_{k_n}\) for all \(n \in\omega\).
    Then, observe that \(F \subseteq \gamma(\mathscr U_{\ell_{k_n}},E) \cap G_{E,\ell_{k_n}}\) for all \(n \in \omega\).
    Thus, since the \(G_{E,n}\) are descending, \(F = E\).

    Now, behold that, as
    \[\mathbf W(E_{k_0} , \mathscr U_{\ell_{k_0}}, \ell_{k_0}) \cap
    \mathbf W\left( F_{{k_0}}, \mathscr U_{\ell_{k_0}} , \ell_{k_0}\right)
    \neq \emptyset,\]
    there is some \(H \in \mathrm{PR}(X)\) so that
    \[H \in \left[E_{k_0} , \gamma(\mathscr U_{\ell_{k_0}},E_{k_0}) \cap G_{E_{k_0},\ell_{k_0}}\right] \cap
    \left[F_{{k_0}} , \gamma(\mathscr U_{\ell_{k_0}} , F_{{k_0}} ) \cap G_{{F_{{k_0}}},\ell_{k_0}} \right].\]
    Hence,
    \begin{align*}
        F = E = E_{k_0} \subseteq H &\subseteq
        \gamma(\mathscr U_{\ell_{k_0}} , F_{{k_0}} ) \cap G_{{F_{{k_0}}},\ell_{k_0}}\\
        &\subseteq \gamma(\mathscr U_{\ell_{k_0}} , F_{{k_0}} )\\
        &= \overrightarrow{T}_{\mathrm{II},\ell_{k_0}}(W_{{k_0}},\mathscr U_{\ell_{k_0}}).
    \end{align*}
    Therefore, Corollary \ref{cor:Translation} applies, finishing the proof.
\end{proof}

Let \(\mathcal P_{\mathrm{fin}}(X)\) represent the set \([X]^{<\omega}\) with the subspace topology
induced by the Vietoris topology on compact subsets of \(X\) (see \cite{MichaelHyperspaces} for more on the Vietoris topology).
As noted in \cite{DouwenPR}, the topology on \(\mathrm{PR}(X)\) is finer than the topology on \(\mathcal P_{\mathrm{fin}}(X)\).
\begin{corollary}
    For any regular space \(X\) and \(\square \in \{1,\mathrm{fin}\}\),
    \[
    \mathsf G_\square^*(\mathcal O_{\mathrm{PR}(X)},\mathcal O_{\mathrm{PR}(X)}) \leq_{\mathrm{II}}
    \mathsf G_\square(\mathcal O_{\mathcal P_{\mathrm{fin}(X)}},\mathcal O_{\mathcal P_{\mathrm{fin}(X)}}).
    \]
\end{corollary}
\begin{proof}
    By \cite[Thm. 4.8]{CHVietoris},
    \[\mathsf G_\square(\Omega_X,\Omega_X) \leftrightarrows \mathsf G_\square(\mathcal O_{\mathcal P_{\mathrm{fin}}(X)},\mathcal O_{\mathcal P_{\mathrm{fin}}(X)}).\]
    So then Theorem \ref{thm:FirstBound} applies to yield
    \[
    \mathsf G_\square^*(\mathcal O_{\mathrm{PR}(X)},\mathcal O_{\mathrm{PR}(X)}) \leq_{\mathrm{II}}
    \mathsf G_\square(\mathcal O_{\mathcal P_{\mathrm{fin}(X)}},\mathcal O_{\mathcal P_{\mathrm{fin}(X)}}).
    \]
\end{proof}

In the next example, we see that the inequality in Theorem \ref{thm:FirstBound} cannot, in general, be reversed.
\begin{example}
By \cite[Cor. 4.8]{SakaiStarVersionsMenger}, \(\mathrm{PR}(\mathbb R)\) is not star-Menger.
Since \(\mathbb R\) is \(\omega\)-Menger (which follows from \(\sigma\)-compactness in every finite power, see \cite{JustMillerScheepers}),
this shows that, even when points are \(G_\delta\), the inequality in Theorem \ref{thm:FirstBound} may not reverse.
\end{example}

Due to the use of Proposition \ref{prop:AnotherPseudocharacterBound} in the proof of Theorem \ref{thm:FirstBound},
one may expect an analogous argument to establish an analog to Theorem \ref{thm:FirstBound} for \(\mathrm{PR}_{\mathbb K}(X)\).
However, an important part in the final steps of the argument for the finite subsets case involving isolating a cofinal sequence
on which the \(E_{k_n}\) are constant does not seem to have remedy in the compact case.
We hence leave this as Question \ref{question:PRAnalog} in Section \ref{sec:Questions}.

In general, however,
\[\mathsf G_\square(\mathcal K_X,\mathcal K_X) \leq_{\mathrm{II}}
\mathsf G^\ast_\square(\mathcal O_{\mathrm{PR}_{\mathbb K}(X)},\mathcal O_{\mathrm{PR}_{\mathbb K}(X)})\]
fails, as witnessed by the following example.
\begin{example}
    Let \(X\) be the one-point Lindel{\"{o}}fication of discrete \(\omega_1\).
    Then, by \cite[Ex. 3.24]{CHCompactOpen} and the game duality results of \cite{ClontzDualSelection},
    \(\mathrm{II} \uparrow G_1(\mathcal K_X, \mathcal K_X).\)
    Note that \(\mathrm{PR}_{\mathbb K}(X) = \mathrm{PR}(X)\) since the compact subsets of \(X\)
    are exactly the finite subsets.
    Since \(\{\{\alpha\} : \alpha \in \omega_1 \}\) is a discrete family of open sets,
    we see that \(\mathrm{PR}_{\mathbb K}(X)\) is not star-Lindel{\"{o}}f by Corollary \ref{cor:PRAXStarLindelof}.
    In particular,
    \(\mathrm{II} \not\uparrow \mathsf G^\ast_\square(\mathcal O_{\mathrm{PR}_{\mathbb K}(X)},\mathcal O_{\mathrm{PR}_{\mathbb K}(X)}).\)
\end{example}

\section{Star-Selection Games in Uniform Spaces}

We first recall a few key facts about uniformities.
For more on uniformities and uniform spaces, the authors recommend \cite{Engelking,IsbellUniformSpaces,Kelley}.

\begin{definition}
    Given a set \(X\), a collection \(\mathcal E\) of sets containing the diagonal
    \(\Delta = \{(x,x) : x \in X\}\) is called a \emph{uniformity on \(X\)} if,
    \begin{enumerate}
        \item for all \(E \in \mathcal E\), \(E^{-1} \in \mathcal E\);
        \item for all \(E,F \in \mathcal E\), \(E \cap F \in \mathcal E\);
        \item for all \(E \in \mathcal E\), there is an \(F \in \mathcal E\) so that \(F \circ F \subseteq E\); and
        \item for all \(E \in \mathcal E\), if \(E \subseteq F\), then \(F \in \mathcal E\).
    \end{enumerate}
    Each \(E \in \mathcal E\) is called an \emph{entourage}.
    We refer to \((X,\mathcal E)\) as a uniform space.
\end{definition}

\begin{definition}
    The topology generated by a uniformity \(\mathcal E\) on a set \(X\) is
    \[\mathscr T = \{ U \subseteq X : (\forall x \in U)(\exists E \in \mathcal E)\ E[x] \subseteq U \}\]
    where \(E[x] = \{ y \in X : \langle x, y \rangle \in E \}.\)
\end{definition}

\begin{definition}
    Given a uniform space \((X,\mathcal E)\) and a collection \(\mathscr U\) of subsets of \(X\), \(\mathscr U\)
    is a \emph{uniform cover} of \(X\) (with respect to \(\mathcal E\))
    if there exists \(E \in \mathcal E\) so that \(\{ E[x] : x \in X \}\) is a refinement of \(\mathscr U\).
    We will say a uniform cover is an open uniform cover if it consists of open sets.
    Let \(\mathcal C_{\mathcal E}(X)\) be the collection of all open uniform covers with respect to \(\mathcal E\).
\end{definition}
If \((X,\mathcal E)\) is a uniform space, then
\begin{itemize}
    \item
    \(\mathsf S_{\mathrm{fin}}(\mathcal C_{\mathcal E}(X),\mathcal O_X)\) is known as the
    \emph{uniform-Menger} (with respect to \(\mathcal E\)) property and
    \item
    \(\mathsf S_{1}(\mathcal C_{\mathcal E}(X),\mathcal O_X)\) is known as the
    \emph{uniform-Rothberger} (with respect to \(\mathcal E\)) property.
\end{itemize}

Recall that a collection \(\mathscr U\) \emph{star-refines} a collection \(\mathscr V\) if
\(\{ \mathrm{St}(U,\mathscr U) : U \in \mathscr U\}\) refines \(\mathscr V\).

The following can be established using the comments in \cite[p. 427]{Engelking} along with the fact that entourages that are
open in \(X^2\) with the induced topology form a base for the uniformity \cite[Thm. 6.6]{Kelley}.
\begin{proposition} \label{prop:UniformCovers}
    If \(\mathcal E\) is a uniformity on a non-empty set \(X\), then \(\mathcal C_{\mathcal E}(X)\) has the following properties:
     \begin{itemize}
        \item
        If \(\mathscr U \in \mathcal C_{\mathcal E}(X)\) and \(\mathscr U\) is a refinement of an open cover \(\mathscr V\), then \(\mathscr V \in \mathcal C_{\mathcal E}(X)\).
        \item
        For any \(\mathscr U_1 , \mathscr U_2 \in \mathcal C_{\mathcal E}(X)\), there exists \(\mathscr V \in \mathcal C_{\mathcal E}(X)\) so that \(\mathscr V\) refines
        both \(\mathscr U_1\) and \(\mathscr U_2\).
        \item
        For every \(\mathscr U \in \mathcal C_{\mathcal E}(X)\), there exists \(\mathscr V \in \mathcal C_{\mathcal E}(X)\) so that \(\mathscr V\) is a star-refinement of \(\mathscr U\).
        \item
        For any \(x,y \in X\) with \(x \neq y\), there is a \(\mathscr U \in \mathcal C_{\mathcal E}(X)\) so that,
        for all \(U \in \mathscr U\), \(x \in U \implies y \notin U\) and \(y \in U \implies x \notin U\).
    \end{itemize}
\end{proposition}

The following follows from \cite[Prop. 8.1.16]{Engelking}.
\begin{proposition} \label{prop:EngelkingUniform}
    Suppose \(X\) is a topological space.
    If \(X\) admits a collection of open covers \(\mathcal C\) so that the conditions of Proposition \ref{prop:UniformCovers} hold,
    then \(X\) admits a uniformity \(\mathcal E\) so that \(\mathcal C = \mathcal C_{\mathcal E}(X)\).
\end{proposition}

The following captures \cite[Thm. 1, Thm. 20]{KocinacUniform}, but extends
them to the other strategy types considered in this paper.
\begin{theorem} \label{thm:MainUniformTheorem}
    Let \((X,\mathcal E)\) be a uniform space. Then
    \[
    \mathsf G_{\square}(\mathcal C_{\mathcal E}(X), \mathcal O_X)
    \leftrightarrows \mathsf {SG}_{X, \square}^*(\mathcal C_{\mathcal E}(X), \mathcal O_X)
    \leftrightarrows \mathsf G_{\square}^\ast(\mathcal C_{\mathcal E}(X), \mathcal O_X).
    \]
\end{theorem}
\begin{proof}
    We first show that \(\mathsf G_{\square}(\mathcal C_{\mathcal E}(X), \mathcal O_X) \leq^+_{\mathrm{II}}
   \mathsf {SG}_{X, \square}^*(\mathcal C_{\mathcal E}(X), \mathcal O_X)\).
    Define \(\overleftarrow{R}_{\mathrm{I}} : \mathrm{Gal}(X, \mathcal C_{\mathcal E}(X)) \to \mathcal P\left( \mathcal C_{\mathcal E}(X) \right)\) by
    \[
    \overleftarrow{R}_{\mathrm{I}}(\mathrm{Cons}(X,\mathscr U)) = \{ \mathscr V \in \mathcal C_{\mathcal E}(X) : \mathrm{Cons}(X,\mathscr U) = \mathrm{Cons}(X,\mathscr V)\}.
    \]
    For each \(U \in \mathscr T_X\), choose \(\gamma(U) \in U\) and, for \(V \in \mathscr V \in  \overleftarrow{R}_{\mathrm{I}}(\mathrm{Cons}(X,\mathscr U))\), we define
    \[
    \overrightarrow{R}_{\mathrm{II}}(V, \mathscr V, \mathrm{Cons}(X,\mathscr U)) = \mathrm{St}(\gamma(V), \mathscr V).
    \]
    Clearly, \(\overleftarrow{R}_{\mathrm{I}}(\mathrm{Cons}(X,\mathscr U)) \neq \emptyset\) and is a subset of \(\mathcal C_{\mathcal E}(X)\).
    When
    \[
    V \in \mathscr V \in  \overleftarrow{R}_{\mathrm{I}}(\mathrm{Cons}(X,\mathscr U)),
    \]
    we know that \(\mathrm{Cons}(X,\mathscr U) = \mathrm{Cons}(X,\mathscr V)\) and so \(\mathrm{St}(\gamma(V), \mathscr V) \in \mathrm{Cons}(X,\mathscr U)\).
    Finally, suppose \(\mathcal F_n \subseteq \mathscr V_n \in  \overleftarrow{R}_{\mathrm{I}}(\mathrm{Cons}(X,\mathscr U_n))\)
    are finite so that \(\bigcup_n \mathcal F_n\) is an open cover of \(X\).
    For each \(n\in\omega\) and each \(V \in \mathcal F_n\), \(\gamma(V) \in V \in \mathscr V_n\),
    so we can conclude that \(V \subseteq \mathrm{St}(\gamma(V), \mathscr V_n)\).
    Thus
    \[
    \left\{\overrightarrow{R}_{\mathrm{II}}(V, \mathscr V_n, \mathrm{Cons}(X,\mathscr U_n)) : n \in \omega \wedge V \in \mathcal F_n\right\}
    \]
    forms an open cover as well.
    Hence, Theorem \ref{thm:RelationalTranslation} applies.

    Next, we show that \(\mathsf {SG}_{X, \square}^*(\mathcal C_{\mathcal E}(X), \mathcal O_X)
    \leq^+_{\mathrm{II}} \mathsf G^\ast_{\square}(\mathcal C_{\mathcal E}(X), \mathcal O_X)\).
    To see this, define \(\overleftarrow{R}_{\mathrm{I}} : \mathrm{Gal}(\mathrm{id},\mathcal C_{\mathcal E}(X)) \to \mathcal P(\mathrm{Gal}(X,\mathcal C_{\mathcal E}(X))\)
    by
    \[
    \overleftarrow{R}_{\mathrm{I}}(\mathrm{Cons}(\mathscr U,\mathscr U)) = \{\mathrm{Cons}(X,\mathscr V) :
    \mathscr V \in \mathcal C_{\mathcal E}(X) \wedge \mathrm{Cons}(\mathscr U,\mathscr U) = \mathrm{Cons}(\mathscr V,\mathscr V)\}.
    \]
    For each \(x \in X\) and \(\mathscr U \in \mathcal C_{\mathcal E}(X)\), let \(U_{x,\mathscr U} \in \mathscr U\) be so that \(x \in U_{x,\mathscr U}\).
    For
    \[
    \mathrm{St}(x,\mathscr V) \in \mathrm{Cons}(X,\mathscr V) \in \overleftarrow{R}_{\mathrm{I}}(\mathrm{Cons}(\mathscr U,\mathscr U)),
    \]
    define
    \[
    \overrightarrow{R}_{\mathrm{II}}(\mathrm{St}(x,\mathscr V),\mathrm{Cons}(X,\mathscr V),\mathrm{Cons}(\mathscr U,\mathscr U)) = \mathrm{St}(U_{x,\mathscr V},\mathscr V).
    \]
    Then if \(\mathrm{St}(x,\mathscr V) \in \mathrm{Cons}(\mathscr V,\mathscr V) \in \overleftarrow{R}_{\mathrm{I}}(\mathrm{Cons}(\mathscr U,\mathscr U))\), we know that
    \[
    \overrightarrow{R}_{\mathrm{II}}(\mathrm{St}(x,\mathscr V),\mathrm{Cons}(X,\mathscr V),\mathrm{Cons}(\mathscr U,\mathscr U)) \in \mathrm{Cons}(\mathscr U,\mathscr U).
    \]
    Also, if \(\{\mathrm{St}(F_n,\mathscr V_n) : n \in \omega\}\) forms a cover, where the \(F_n \subseteq X\) are finite, then, as
    \[
    \mathrm{St}(F_n,\mathscr V_n) \subseteq \bigcup_{x \in F_n}\mathrm{St}(U_{x,\mathscr V_n},\mathscr V_n),\]
    we know that
    \[\bigcup_{n\in\omega}\left\{ \overrightarrow{R}_{\mathrm{II}}(\mathrm{St}(x,\mathscr V_n),\mathrm{Cons}(X,\mathscr V_n),\mathrm{Cons}(\mathscr U_n,\mathscr U_n))
    : x \in F_n \right\}\]
    also forms a cover of \(X\).
    So Theorem \ref{thm:RelationalTranslation} applies.

    Finally, we show that \(\mathsf G^\ast_{\square}(\mathcal C_{\mathcal E}(X), \mathcal O_X)
    \leq^+_{\mathrm{II}} \mathsf G_{\square}(\mathcal C_{\mathcal E}(X), \mathcal O_X)\).
    For each \(\mathscr U \in \mathcal C_{\mathcal E}(X)\), use the properties of uniform covers to choose \(\mathscr U^\star \in \mathcal C_{\mathcal E}(X)\)
    which is a star-refinement of \(\mathscr U\).
    As noted in Remark \ref{rmk:Choice}, we can then define
    \(\alpha_{\mathscr U}:\mathscr U^\star \to \mathscr U\) to be so that \(\mathrm{St}(U,\mathscr U^\star) \subseteq \alpha_{\mathscr U}(U)\)
    for all \(U \in \mathscr U^\star\)
    and \(\alpha_{\mathscr U}(U) = \alpha_{\mathscr U}(V)\) whenever \(\mathrm{St}(U,\mathscr U^\star) = \mathrm{St}(V,\mathscr U^\star)\).
    We then define \(\overleftarrow{T}_{\mathrm{I}} : \mathcal C_{\mathcal E}(X) \to \mathrm{Gal}(\mathrm{id},\mathcal C_{\mathcal E}(X))\) by
    \(\overleftarrow{T}_{\mathrm{I}}(\mathscr U) = \mathrm{Cons}(\mathscr U^\star,\mathscr U^\star)\) and
    \[
     \overrightarrow{T}_{\mathrm{II}}(\mathrm{St}(U,\mathscr U^\star), \mathscr U) = \alpha_{\mathscr U}(U).
    \]
    Note that \(\overrightarrow{T}_{\mathrm{II}}\) is well-defined by the conditions on \(\alpha_{\mathscr U}\).
    For \(W \in \overleftarrow{T}_{\mathrm{I}}(\mathscr U)\), there is \(U \in \mathscr U^\star\) so that \(W = \mathrm{St}(U,\mathscr U^\star)\).
    Then observe that \[\overrightarrow{T}_{\mathrm{II}}(W, \mathscr U) = \alpha_{\mathscr U}(U) \in \mathscr U.\]
    To finish the proof, suppose we have
    \[\mathcal F_n \in \left[ \overleftarrow{T}_{\mathrm{I}}(\mathscr U_n) \right]^{<\omega}\]
    so that \(\bigcup_{n\in\omega} \mathcal F_n\) forms a cover of \(X\).
    By the fact that applying \(\overrightarrow{T}_{\mathrm{II}}\) produces a collection of open sets which is refined by
    \(\bigcup_{n\in\omega} \mathcal F_n\), we see that Corollary \ref{cor:Translation} applies.
\end{proof}
The following result can be seen as a generalization of \cite[Thm. 3.4 (1)]{AlamChandra} beyond the paracompact case and
to the other strategy types discussed in this paper.
\begin{theorem} \label{thm:UniformKMenger}
    For any uniform space \((X,\mathcal E)\),
    \[\mathsf {SG}_{X, \mathrm{fin}}^*(\mathcal C_{\mathcal E}(X), \mathcal O_X)
    \leftrightarrows \mathsf {SG}_{\mathbb K, 1}^*(\mathcal C_{\mathcal E}(X), \mathcal O_X)
    \leftrightarrows \mathsf {SG}_{\mathbb K, \mathrm{fin}}^*(\mathcal C_{\mathcal E}(X), \mathcal O_X).\]
\end{theorem}
\begin{proof}
    By Proposition \ref{prop:starSingleVsFinite},
    \[\mathsf {SG}_{\mathbb K, 1}^*(\mathcal C_{\mathcal E}(X), \mathcal O_X)
    \leftrightarrows \mathsf {SG}_{\mathbb K, \mathrm{fin}}^*(\mathcal C_{\mathcal E}(X), \mathcal O_X)\]
    and, by Corollary \ref{cor:Monotonicity},
    \[\mathsf {SG}_{X, \mathrm{fin}}^*(\mathcal C_{\mathcal E}(X), \mathcal O_X)
    \leq_{\mathrm{II}}^+ \mathsf {SG}_{\mathbb K,\mathrm{fin}}^*(\mathcal C_{\mathcal E}(X), \mathcal O_X).\]
    To finish the proof, we show that
    \[\mathsf {SG}_{\mathbb K, \mathrm{fin}}^*(\mathcal C_{\mathcal E}(X), \mathcal O_X)
    \leq_{\mathrm{II}}^+ \mathsf {SG}_{X, \mathrm{fin}}^*(\mathcal C_{\mathcal E}(X), \mathcal O_X).\]
    For \(\mathscr U \in \mathcal C_{\mathcal E}(X)\), let \(\gamma(\mathscr U) \in \mathcal E\) be so that
    be so that
    \[\left\{ \left(\gamma(\mathscr U) \circ \gamma(\mathscr U)
    \circ \gamma(\mathscr U)\right)[x] : x \in X \right\}\]
    refines \(\mathscr U\).
    Then define \(\overleftarrow{T}_{\mathrm{I}}(\mathscr U) = \{ \gamma(\mathscr U)[x] : x \in X \}\).

    We can then define \(\eta(K,\mathscr U)\) for \(K \in K(X)\) and \(\mathscr U \in \mathcal C_{\mathcal E}(X)\)
    to be a finite subset of \(K\)
    so that \[K \subseteq \bigcup \left\{ \gamma(\mathscr U)[x] : x \in \eta(K, \mathscr U) \right\}.\]
    Then, for \(\mathcal F \in [K(X)]^{<\omega}\) and \(\mathscr V , \mathscr U \in \mathcal C_{\mathcal E}(X)\), we define
    \[\overrightarrow{T}_{\mathrm{II}}\left( \left\{ \mathrm{St}(K, \mathscr V) : K \in \mathcal F \right\}, \mathscr U \right)
    =\bigcup_{K \in \mathcal F} \left\{ \mathrm{St}(x, \mathscr U) : x \in \eta(K, \mathscr U) \right\}.\]

    To finish this application of Theorem \ref{thm:GeneralTranslation}, we just need to show that \(\overrightarrow{T}_{\mathrm{II}}\)
    takes winning plays to winning plays.
    So let \(\langle \mathscr U_n : n \in \omega \rangle \in \mathcal C_{\mathcal E}(X)^\omega\),
    \(\mathscr V_n = \overleftarrow{T}_{\mathrm{I}}(\mathscr U_n)\), and
    suppose we have \(\mathcal F_n \in [K(X)]^{<\omega}\) so that
    \[\bigcup_{n\in\omega} \left\{ \mathrm{St}(K,\mathscr V_n) : K \in \mathcal F_n \right\} \in \mathcal O_X.\]
    Let \(x \in X\) be arbitrary and choose \(n \in \omega\) and \(K \in \mathcal F_n\) so that \(x \in \mathrm{St}(K,\mathscr V_n)\).
    Then there is some \(w_0 \in X\) so that \(x \in \gamma(\mathscr U_n)[w_0]\) and \(K \cap \gamma(\mathscr U_n)[w_0] \neq\emptyset\).
    Let \(w_1 \in K \cap \gamma(\mathscr U_n)[w_0]\) and \(y \in \mathcal F_n\) be so that \(w_1 \in \gamma(\mathscr U_n)[y]\).
    Then \(\langle x,w_0 \rangle, \langle w_0,w_1\rangle , \langle w_1,y \rangle \in \gamma(\mathscr U_n)\) which implies that
    \(\langle x, y \rangle \in \gamma(\mathscr U_n) \circ \gamma(\mathscr U_n) \circ \gamma(\mathscr U_n)\).
    By the assumption on \(\gamma(\mathscr U_n)\), there is some \(U \in \mathscr U_n\) so that
    \((\gamma(\mathscr U_n) \circ \gamma(\mathscr U_n) \circ \gamma(\mathscr U_n))[x] \subseteq U\)
    and so we see that \(y \in U\).
    As \(x \in U\) as well, we see that \(x \in \mathrm{St}(y , \mathscr U_n)\).
    This finishes the proof.
\end{proof}

Theorem \ref{thm:UniformKMenger} fails for \(\mathsf {SG}_{X, 1}^*(\mathcal C_{\mathcal E}(X), \mathcal O_X)\),
as illustrated by the following example.
\begin{example} \label{ex:UniformReals}
    Consider \(X = \mathbb R\) with the uniformity \(\mathcal E\) generated by the standard metric;
    that is, the uniformity generated by
    \[\{\{ (x,y) \in \mathbb R^2 : |x-y| < \varepsilon\} : \varepsilon > 0\}.\]
    Then Two has a winning Markov strategy in \(\mathsf {SG}_{\mathbb K, 1}^*(\mathcal C_{\mathcal E}(X), \mathcal O_X)\)
    but One has a predetermined winning strategy in \(\mathsf {SG}_{X, 1}^*(\mathcal C_{\mathcal E}(X), \mathcal O_X)\).
    To see this, note that a winning Markov strategy for Two in \(\mathsf {SG}_{\mathbb K, 1}^*(\mathcal C_{\mathcal E}(X), \mathcal O_X)\)
    follows from \(\sigma\)-compactness of \(\mathbb R\), as illustrated in Example \ref{ex:starSingleFinite}.
    A predetermined winning strategy for One in \(\mathsf {SG}_{X, 1}^*(\mathcal C_{\mathcal E}(X), \mathcal O_X)\)
    is \[\sigma(n) := \left\{B\left(x,\frac{1}{2^{n+2}}\right) : x \in \mathbb R \right\}.\]
    Then, as \(\mathrm{St}(x,\sigma(n)) \subseteq B\left(x,\frac{1}{2^n}\right)\), there is no sequence
    \(\langle x_n : n \in \omega \rangle\) or \(\mathbb R\) so that
    \(\{\mathrm{St}(x_n,\sigma(n)) : n \in \omega\}\) covers \(\mathbb R\).
\end{example}

\begin{definition}
    Given a space \(X\), the \emph{universal} (or \emph{fine}) \emph{uniformity} on \(X\) is the finest uniformity
    on \(X\) compatible with its topology (see \cite[Exercise 8.1.C]{Engelking} and \cite[Thm. 1.20]{IsbellUniformSpaces}).
\end{definition}

The following, along with Theorem \ref{thm:MainUniformTheorem}, extends \cite[Thm. 3.2 (1), (2)]{AlamChandra}
and the paracompact case of \cite[Thm. 3.3 (1)]{AlamChandra} to more strategy types.
\begin{theorem} \label{thm:MengerParacompact}
    If \(X\) is paracompact and \(\mathfrak U\) is the universal uniformity on \(X\), then
    \[\mathsf{G}_\square(\mathcal{C}_{\mathfrak U}(X), \mathcal O_X) \leftrightarrows \mathsf{G}_\square(\mathcal O_X,\mathcal O_X).\]
\end{theorem}
\begin{proof}
    By \cite[Thm 5.1.12]{Engelking}, every open cover of a paracompact space has an open star-refinement.
    Then, by Proposition \ref{prop:EngelkingUniform}, \(\mathcal O_X\) generates a uniformity on \(X\),
    the universal uniformity \(\mathfrak U\), and \(\mathcal O_X = \mathcal C_{\mathfrak U}(X)\).
\end{proof}

Using the well-known results of Hurewicz and Pawlikowski (see \cite{Pawlikowski}) along with Theorem \ref{thm:MengerParacompact}, we obtain the following.

\begin{corollary} \label{cor:Pawlikowski}
    If \(X\) is paracompact and \(\mathfrak U\) is the universal uniformity on \(X\), then
    \[\mathrm{I} \underset{\mathrm{pre}}{\uparrow} \mathsf{G}_\square(\mathcal{C}_{\mathfrak U}(X), \mathcal O_X)
    \iff \mathrm{I} \uparrow \mathsf{G}_\square(\mathcal{C}_{\mathfrak U}(X), \mathcal O_X).\]
\end{corollary}

In general, Theorem \ref{thm:MengerParacompact} cannot be generalized beyond paracompact spaces,
as the following example illustrates.
(Compare to \cite[Note 4]{KocinacUniform}.)
\begin{example}
    Let \(\mathcal E\) be any uniformity on \(\omega_1\) compatible with its order topology.
    Note that \(\omega_1\) is not Menger as it's not even Lindel{\"{o}}f.
    However, \(\omega_1\) is uniform-Rothberger (hence, uniform-Menger) with respect to \(\mathcal E\).

    To establish that \(\omega_1\) is uniform-Rothberger, we first show that self compositions of open entourages on
    \(\omega_1\) create open sets that contain tails.
    So let \(U \subseteq \omega_1^2\) be open so that \(\Delta \subseteq U\).
    For each limit ordinal \(\lambda < \omega_1\), let \(b_\lambda < \lambda\) be so that
    \((b_\lambda , \lambda]^2 \subseteq U\).
    By Fodor's Pressing Down Lemma, there is a stationary, thus cofinal, set \(S \subseteq \omega_1\) and some
    \(\beta < \omega_1\) so that \(b_\lambda = \beta\) for all \(\lambda \in S\).

    We now show that \([ \beta+1 , \omega_1) \subseteq (U\circ U)[\beta+1]\).
    Indeed, let \(\gamma > \beta\) and \(\lambda \in S\) be so that \(\gamma < \lambda\).
    Notice that \(b_\lambda = \beta < \gamma < \lambda\) since \(\lambda \in S\);
    so we obtain that \(( \lambda , \gamma ) \in U\) and that \((\beta+1, \lambda)\in U\).
    Hence, \((\beta + 1 , \gamma) \in U \circ U\), which implies that \(\gamma \in (U \circ U)[\beta+1]\).

    Now, this means that, for any uniformity \(\mathcal E\) on \(\omega_1\)
    compatible with its order topology, every uniform cover of \(\omega_1\) contains a co-countable element;
    so \(\omega_1\) is uniform-Rothberger.
\end{example}

\section{Open Questions} \label{sec:Questions}

\begin{question}
    Are there any other interesting selection principles of the form \[\mathsf S_\square(\mathrm{Gal}(f, \mathcal A), \mathcal B)\]
    where \(\square \in \{1,\mathrm{fin}\}\), \(\mathcal A\) and \(\mathcal B\) are collections, and \(f\) is not either constant
    or the identity?
\end{question}
\begin{question} \label{question:PRAnalog}
For any space \(X\), is there an analog to Theorem \ref{thm:FirstBound} for \(\mathrm{PR}_{\mathbb K}(X)\)?
That is, is it the case that
\[\mathsf G^\ast_\square(\mathcal O_{\mathrm{PR}_{\mathbb K}(X)},\mathcal O_{\mathrm{PR}_{\mathbb K}(X)}) \leq_{\mathrm{II}}
\mathsf G_\square(\mathcal K_X, \mathcal K_X)\text{?}\]
\end{question}
\begin{question}
    Can Corollary \ref{cor:Pawlikowski} be generalized beyond the paracompact setting, despite the fact that Theorem \ref{thm:MengerParacompact}
    cannot?
\end{question}

\providecommand{\bysame}{\leavevmode\hbox to3em{\hrulefill}\thinspace}
\providecommand{\MR}{\relax\ifhmode\unskip\space\fi MR }
\providecommand{\MRhref}[2]{%
  \href{http://www.ams.org/mathscinet-getitem?mr=#1}{#2}
}
\providecommand{\href}[2]{#2}

\end{document}